\documentclass[12pt]{article}
\usepackage[margin=1in]{geometry}

\usepackage{graphicx} 
\usepackage{amsfonts}       
\usepackage{nicefrac}       
\usepackage{microtype}      
\usepackage{algorithm}
\usepackage{algorithmic}
\usepackage{setspace}
\usepackage{caption}
\usepackage{amsfonts}
\usepackage{amsmath}
\usepackage{amssymb}
\usepackage{upgreek}
\usepackage{fancyhdr}
\usepackage{titlesec}
\usepackage{indentfirst}
\usepackage{booktabs}
\usepackage{verbatim}
\usepackage{color}
\usepackage{tcolorbox}
\usepackage{mathtools}
\usepackage{bm}
\usepackage{amsthm}
\usepackage{wasysym}
\usepackage{empheq}
\usepackage{hyperref}
\usepackage{caption}
\usepackage{subcaption}
\usepackage{ulem}
\usepackage{enumerate}
\usepackage{pgfplots}
\usepgfplotslibrary{groupplots}
\pgfplotsset{compat=1.18}

\newtheorem{theorem}{Theorem}[section]
\newtheorem{corollary}{Corollary}

\newtheorem{lemma}[theorem]{Lemma}

\newtheorem{assumption}{Assumption}

\newtheorem{definition}{Definition}
\newtheorem{remark}{Remark}[section]

\def\e{\epsilon}

\def\lam{\lambda}

\def\th{\theta}

\def\th{\theta}

\def\O{\Omega}

\def\E{\mathbb E}

\def\R{\mathbb R}
\def\S{{\mathbb S}}

\def\mL{\mathcal{L}_{b,\Sigma}}

\def\mA{\mathcal A}

\def\l{\left}
\def\r{\right}
\def\la{\left\langle}
\def\ra{\right\rangle}
\def\ll{\left\lVert}
\def\rl{\right\rVert}

\def\({\left(}
\def\){\right)}

\usepackage[comma, sort]{natbib}
\setcitestyle{author-year}
\hypersetup{
    colorlinks=true,
    linkcolor=blue,
    filecolor=blue,      
    urlcolor=blue,
    citecolor=cyan,
}
\def\pt{\partial}
\def\nb{\nabla}

\def\t{\tilde}

\def\dt{{\Delta t}}

\def\argmax{\text{argmax}}
\title{PhiBE-Q-Learning: Bridging Off-Policy Reinforcement Learning and Continuous-Time Control}
\author{ Yuhua Zhu \thanks{Department of Statistics and Data Science, University of California, Los Angeles, USA. \texttt{yuhuazhu@ucla.edu}}  }

\date{}
\begin{document}

\maketitle
\begin{abstract}
In this paper, we develop an off-policy method for continuous-time reinforcement learning (CTRL), where the system dynamics are governed by an unknown stochastic differential equation (SDE) and only discrete-time trajectory data are available. 
A central challenge is that the classical discrete-time state--action value function loses its action-dependent signal in the continuous-time limit \citep{tallec2019making, jia2023q, baird1994reinforcement}. On the other hand, continuous-time control provides local notions such as the instantaneous advantage function $q(s,a)$, but these typically rely on state value function $V(s)$. To obtain a formulation with a single unknown function, we introduce a new definition of the state--action value function $Q(s,a)$ in CTRL and derive its governing equation.

Building on the PhiBE approximation \citep{zhu2024phibe,zhu2025optimal}, we propose iterative algorithms to approximate the optimal $Q$-function in both model-based and model-free settings using only discrete-time off-policy data. Under linear function approximation, we establish error bounds and explicit convergence rates for the algorithm in the popoluation level. The finite-data algorithm is evaluated numerically.
\end{abstract}

\section{Introduction}






Reinforcement learning (RL) has achieved remarkable success in the digital world, including AlphaGo \citep{silver2016mastering}, Atari gameplay \citep{mnih2015human}, and the fine-tuning of large language models \citep{Ziegler2019}. A key feature underlying these successes is that the system state remains static between observations. This property is natural in digital settings, making such environments well suited to RL algorithms built on the Markov decision process (MDP) framework. However, decision-making in the physical world departs from this digital paradigm in two fundamental respects. First, the objective is not merely to maximize rewards at the discrete times when data are observed, but to control performance continuously over time, including during periods without observations. In dynamic treatment regimes, for example, a patient's physiological state (e.g., glucose level) evolves continuously, while measurements are only taken intermittently. The clinical goal, however, is to regulate the state at all times \citep{emerson2023offline, cobelli2009diabetes, battelino2019clinical}. RL in the physical world thus poses a substantially more challenging objective than its classical counterpart. Second, many physical systems exhibit structured and smooth continuous-time dynamics, a feature that is largely absent in digital environments and not exploited by the MDP framework. RL in the physical world is also known as continuous-time reinforcement learning (CTRL) \citep{wang2020reinforcement}.

Off-policy learning is equally critical in real-world applications \citep{watkins1992q, sutton1998reinforcement, ma2020off}. Off-policy data consists of trajectories generated under a behavior policy that differs from the target policy, and the behavior policy may be unknown. Leveraging off-policy data can substantially improve sample efficiency, particularly in domains where data collection is costly, such as healthcare. Moreover, it enables the use of existing datasets, which is critical in settings where exploration is unsafe, such as autonomous driving.

The central question this paper studies is: what is a principled framework for CTRL when the system evolves according to an unknown stochastic differential equation (SDE), yet only discrete-time off-policy observations are available?

Existing work has largely followed two directions. The first approach seeks to reconstruct the underlying continuous-time dynamics from discrete observations and subsequently solve the induced optimal control problem using tools from stochastic control \citep{pmlr-v125-agarwal20b,pmlr-v139-yildiz21a,yang2014reinforcement,kamalapurkar2016model}. This approach has two appealing features. First, the system dynamics depend only on the state and action, therefore, it is independent of the policy, making it naturally suited for leveraging off-policy data. Second, it preserves the continuous-time structure of the problem. However, this approach also faces two fundamental limitations. First, identifying a continuous-time model from discrete observations is inherently ill-posed: in general, infinitely many continuous-time dynamics can induce the same discrete-time transition law \citep{zhu2025optimal}. As a result, model misspecification is unavoidable in practice and may lead to suboptimal policies. Second, the approach decomposes the problem into model identification and control optimization, each introducing its own source of error. When the ultimate goal is to find the optimal policy, explicitly estimating the underlying dynamics may be unnecessary and potentially inefficient.

The second approach instead reformulates CTRL as a discrete-time MDP and applies standard model-free off-policy RL algorithms, such as Q-learning or conservative Q-learning \citep{watkins1992q,kumar2020conservative}. These formulations rely solely on discrete-time transition dynamics, thereby avoiding the unidentifiability issues inherent in the first approach. In addition, model-free RL methods directly learn the optimal policy without explicitly estimating the underlying dynamics, providing a simple and broadly applicable framework across different systems. However, this approach also has two notable limitations. First, discrete-time RL algorithms do not exploit the SDE structure of the underlying dynamics. Instead, they treat all systems the same, leading to inefficient use of information and potentially suboptimal policies. More specifically, the learning error can be highly sensitive to both system parameters and reward functions \citep{zhu2025optimal}. Second, off-policy learning in the MDP framework is typically based on the action–value function $Q(s,a)$ \citep{watkins1992q}. As observed in \citep{jia2023q, tallec2019making, baird1994reinforcement}, in the continuous-time setting, the state-action $Q(s,a)$ function can be viewed as a small perturbation around the state value function $V(s)$, implying that its variation across actions is small, which makes it difficult to accurately distinguish between actions. To address this issue, these papers propose a rescaled advantage function $q(s,a)$ to amplify variation in the action space. However, the off-policy algorithms developed in these papers requires  jointly estimating the advantage  $q$ and state value $V$. HHere we instead seek a closed equation in a single unknown state–action function
(see Section~\ref{sec:off-policy-rl} for a detailed discussion).

We tackle this challenge in two steps. First, we introduce a new definition of the state--action value function $Q$. In classical RL, a key advantage of the $Q$-function is its ability to handle off-policy data directly, whereas the value function $V$ does not share this property. This distinction arises because the discrete-time Bellman equation for $Q$ depends only on the transition distribution under a given action $a$, while the Bellman equation for $V$ depends on the transition induced by a policy $\pi$. Our first contribution is to define a continuous-time $Q$-function that preserves both advantages from RL $Q$ and continuous-time $q$. Specifically, the proposed $Q$-function satisfies two properties: (i) it does not degenerate to the state value function $V$, unlike standard continuous-time limits of RL formulations; and (ii) its governing equation involves action-conditioned dynamics that are independent of the policy $\pi$; and
(iii) the equation is closed in $Q$, without requiring a separate unknown state value function $V$.

The second step is to develop a new iterative algorithm to approximate the optimal value function $Q^*$. We begin with an algorithm for the setting where the underlying dynamics are known, and provide error bounds and explicit convergence rates for the algorithm in the popoluation level. Building on this, and leveraging the PhiBE framework~\citep{zhu2024phibe,zhu2025optimal}, we propose a continuous-time analogue of model-free Q-learning~\citep{watkins1992q}. This is the second key contribution of this paper. The proposed algorithm for off-policy continuous-time RL combines the strengths of continuous-time control and classical RL while avoiding their key limitations. Compared to the model-based control approaches, it relies only on discrete-time transition data and does not require explicit estimation of the underlying dynamics. It is model-free in the sense that it updates $Q$ directly from transition samples without separately fitting a dynamics model. Compared to standard RL algorithms, it explicitly exploits the SDE structure of the dynamics, integrating discrete-time observations with continuous-time evolution. Moreover, the resulting $Q$-function is not a vanishing perturbation of the value function $V$, which preserves the action-dependent signal in the continuous-time limit.

\textbf{Contributions}. We summarize the main contributions as follows.
\begin{itemize}
    \item We introduce a new state--action value function $Q$ for CTRL and derive its governing equation, which enables off-policy learning in continuous time with a single unknown function.
    \item We develop an iterative algorithm to approximate the optimal state--action value function under known dynamics, and establish error bounds and convergence rates under linear function approximation.
    \item 
    When the dynamics are unknown, we propose \textit{PhiBE-Q-Learning}, a model-free algorithm that exploits the SDE structure of CTRL problems to make effective use of off-policy data. The algorithm accommodates one-step transition data and requires parametrizing only the $Q$ function itself.
\end{itemize}

\subsection{Related work}
\subsubsection{Standard Discrete-time RL}

Classical reinforcement learning (RL) methods such as Q-learning \citep{watkins1992q}, DQN \citep{mnih2015human}, provide a powerful framework for off-policy learning in discrete-time setting. Our approach draws inspiration from these algorithms. Our update rule is an analogue of Q-learning \citep{watkins1992q} in continuous-time. 

In discrete-time Q-learning, convergence typically relies on the discount factor $\gamma<1$, which ensures that the optimal Bellman operator is a contraction mapping \citep{bertsekas2012dynamic}. In the continuous-time setting, the corresponding discount factor takes the form $\gamma = e^{-\beta \Delta t} \to 1$ as the time discretization $\Delta t \to 0$. This limit would, in principle, destroy the contraction property and lead to ill-conditioned or diverging guarantees. In this paper, we refine the error bounds in terms of $\Delta t$. In particular, we establish convergence and stability results that are uniform in $\Delta t$. We show that, when the underlying dynamics follow an SDE, the additional regularity of the system yields well-conditioned guarantees even $\dt\to0$. Similar results are also derived in other CTRL papers under different settings \citep{zhu2024phibe,zhu2025optimal,mou2024bellman}.

\subsubsection{Continuous-time RL(CTRL)}
A number of recent works have investigated continuous-time reinforcement learning (CTRL), where the system dynamics are modeled by SDEs \citep{jia2022policy-eval, jia2022policy-gradient, szpruch2024optimal, wang2020continuous, kim2020hamilton}.
In particular, \citep{wang2020reinforcement} connects CTRL with stochastic control with several extensions developed in subsequent works \citep{jia2022policy-gradient, szpruch2024optimal}.
These methods typically enjoy strong theoretical guarantees and exploit the SDE structure when continuous-time data are available. However, they face three major limitations in settings with discrete-time off-policy data. First, when applied to discrete-time observations, these algorithms do not fully exploit the SDE structure of the underlying dynamics. Second, many approaches depend on long trajectories spanning from an initial state to a terminal state, which restricts applicability when only fragmented or short trajectories are available. Finally, the $q$ function defined in \citep{jia2023q,tallec2019making, baird1994reinforcement} remains dependent on the value function $V$, which requires joint parametrizing $q$ and $V$ in the algorithm. These limitations motivate the development of methods that can learn directly from discrete observations, and leverage off-policy samples in CTRL with one single parametrization.

\subsubsection{PhiBE}
To enable learning directly from discrete trajectory data, \citep{zhu2024phibe} proposed the PhiBE framework, which explicitly accounts for discretization effects and leverages the SDE structure. By doing so, PhiBE improves the accuracy and efficiency of learning from discrete trajectory data in stochastic continuous-time systems. At the same time, PhiBE still leaves certain aspects unaddressed. In particular, existing PhiBE algorithms do not fully support learning from off-policy data. This work extends the PhiBE framework to address the off-policy data utilization. Our method aims to find the solution of PhiBE instead of the original HJB equation, thereby retaining PhiBE's key advantage of handling the discrete trajectory data. Moreover, by introducing a new definition of the continuous-time action-value function $Q$, we derive an analogue of the optimal Bellman equation on $Q$ independent of policy $\pi$. This formulation allows the algorithm to utilize trajectories generated under different policies. 

Finally, we emphasize that our attention in this paper is restricted to the convergence rate in the population level. The finite sample error \citep{zhao2025sample,muehlebach2025sample,zhu2024phibe} is left for future work.

\subsubsection{Classical HJB solvers}

Our iterative method under known dynamics shares the same intuition as classical numerical schemes for HJB equation \citep{schaeffer2016accelerated, feng2013recent}. However, the key difference lies in the formulation and convergence mechanism. Traditional finite-difference methods approximate the HJB equation by discretizing the state space on a grid, which requires careful mesh design and often imposes CFL-type constraints on the grid spacing to ensure stability and convergence. In contrast, our approach avoids explicit spatial discretization by approximating the Galerkin projection directly from trajectory data. This eliminates the need for grid-based computation and CFL conditions, making the method more flexible and better suited for data-driven implementations in CTRL.

Moreover, based on the PhiBE framework, our method could extend to a model-free algorithm when the dynamics are unknown. Specifically, it can directly approximate the action-value function without requiring explicit knowledge of the system dynamics. Here ``model-free" means that transition samples are used directly to update $Q$, without separately estimating a dynamics model.

\subsection{Notation and Organization}
\paragraph{Notation} We define weighted $L^2$-norm as follows
\[
\la f, g\ra = \int f(s,a) g(s,a) \mu(s,a) dsda, \quad \ll f \rl^2 =\la f, f \ra.
\]
Here, $\mu(s,a)$ denotes the probability density of the off-policy data distribution, supported on a compact set $\Omega \subset \mathbb{R}^{d}\times \mathbb{R}^{m}$.

For vector $x\in \R^n$, and function $f(s,a)$, we define $\ll \cdot \rl_2, \ll \cdot \rl_\infty$ as follows,
\[
\ll x \rl^2_2 = \sum_{i=1}^n x_i^2, \quad \ll f \rl_\infty = \sup_{(s,a)\in\O} \ll f \rl_2
\]

If no specific remarks are given, then $\nb, \nb^2$ refer to $\nb_s, \nb^2_s$.

\paragraph{Organization.}
Section~\ref{sec:preliminary} introduces the problem setup. 
Section~\ref{sec:off-policy-rl} presents the new state--action $Q$-function. 
Section~\ref{sec:off-policy-known} develops the algorithm under known dynamics and establishes its theoretical guarantees, while Section~\ref{sec:off-policy-phibe} derives the model-free algorithm for unknown dynamics. 
Section~\ref{sec:4.3} reports numerical experiments that validate the theory.

\section{Preliminary}\label{sec:preliminary}
\subsection{Stochastic optimal control}
We consider a continuous-time stochastic optimal control problem defined on a filtered probability space $(\Omega, \mathcal{F}, \{\mathcal{F}_t\}, \mathbb{P})$. The state process $\{s_t\}_{t \ge 0}$ takes values in $\mathbb{R}^d$ and evolves according to a controlled stochastic differential equation (SDE):
\begin{equation}\label{dyn}
    ds_t = b(s_t, a_t)\, dt + \sigma(s_t, a_t)\, dB_t,
\end{equation}
where $b: \mathbb{R}^d \times \mathcal{A} \to \mathbb{R}^d$ is the drift term, $\sigma: \mathbb{R}^d \times \mathcal{A} \to \mathbb{R}^{d \times k}$ is the diffusion term, and $B_t$ is a $k$-dimensional standard Brownian motion.  
At time $t$, the agent takes an action $a_t =\pi(s_t)$ according to a feedback control. Here the feedback control is defined as a deterministic mapping 
$\pi : \mathbb{R}^d \to \mathcal{A}$ from the state space to a compact action space $\mathcal{A} \subseteq \mathbb{R}^m$.  
Given a fixed policy $\pi$ and initial state $s_0 = s$, the value function is defined as
\begin{equation}\label{def of V-pi}
    V^\pi(s) = \mathbb{E}\left[ \int_0^\infty e^{-\beta t} r(s_t, \pi(s_t))\, dt \;\middle|\; s_0 = s \right],    
\end{equation}
where $\beta > 0$ is the discount rate. The goal of stochastic optimal control problem is to find the optimal feedback control that maximizes the value function,
\begin{equation}\label{sop-goal}
    \pi^*(s) = \argmax_\pi\, V^\pi(s), \quad V^*(s) = V^{\pi^*}(s).
\end{equation}
We assume the following assumption to ensure the well-posedness of the above stochastic control problem. 
\begin{assumption}
\label{ass}
We write
$C_{\mathrm{pol}}^{4,\alpha}(\mathbb R^d\times\mathcal A)$
for the class of functions that admit locally $C^{4,\alpha}$ extensions to an open neighborhood of $\mathbb R^d\times\mathcal A$ and whose derivatives through order four have at most polynomial growth.

We assume that 
\[
b,\sigma,r
\in
C_{\mathrm{pol}}^{4,\alpha}
(\mathbb R^d\times\mathcal A)
\]
for some $\alpha\in(0,1)$. Moreover, $b$ and $\sigma$ are globally Lipschitz continuous in $(s,a)$, $r$ is uniformly bounded, and there exists $\lambda>0$ such that
\[
\sigma(s,a)\sigma(s,a)^\top
\succeq
\lambda I_d
\]
for every $(s,a)\in\mathbb R^d\times\mathcal A$.
\end{assumption}
Under the above assumption, the optimal value function $V^*(s)$  defined in \eqref{sop-goal} is the unique bounded continuous viscosity solution of the HJB equation
\begin{equation}\label{eq:hjb}
\beta V^*(s) = \sup_{a\in\mA}[r(s,a) + (\mL V^*)(s,a)],\quad 
\end{equation}
Here
\begin{equation}\label{eq:PE}
\mL = b(s,a)\cdot \nabla_s + \frac{1}{2}\Sigma(s,a) :\nabla_s^2
\end{equation}
with $\Sigma(s,a) = \sigma(s,a)\sigma^\top(s,a) \in \R^{d\times d}$ and $\Sigma(s,a) : \nabla_s^2 = \Sigma_{i,j}\Sigma(s,a)_{ij}\partial_{s_i}\partial_{s_j}$. The optimal policy $\pi^*$ is
\[
\pi^*(s) \in \arg\max_{a\in\mA}[r(s,a) + (\mL V^*)(s,a)].
\]
We refer readers to \citep{FS06, YZ99, P09} for detailed accounts of the classical stochastic control theory.

\subsection{Reinforcement learning}
In the setting of reinforcement learning, the dynamics of the environment, i.e., $b(s,a), \sigma(s,a)$ are unknown.
We assume access to discretized information, either in the form of transition densities
\[
    \rho_{\Delta t}(s' \mid s,a), \qquad \mathbb{R}^d \times \mathbb{R}^d \times \mathcal{A} \to [0,\infty),
\]
which denotes the probability density function of $s_{\Delta t}$ given $s_0 = s$ and $a(\tau) = a$ for $\tau \in (0,\Delta t)$, or a dataset of transitions
\begin{equation}\label{data}
\{(s_i, a_i, s'_i, r_i)\}_{i=1}^n, \qquad (s_i, a_i)\sim \mu(s,a), \quad s'_i \sim \rho_{\Delta t}(\,\cdot \mid s_i,a_i), \quad r_i = r(s_i, a_i).    
\end{equation}
Here the actions $\{a_i\}_{i=1}^n$ are generated by some (possibly unknown) behavior policy and are not controlled by the learner. Note that we consider a piecewise constant control where the action is frozen at the left endpoint, i.e., $a_t = a_0$ for $t \in [0,\Delta t]$ in the data set, while the goal is still to find the optimal feedback control $\pi^*(s)$ defined in \eqref{sop-goal}.

\subsection{PhiBE}\label{sec:2.2}
As indicated in \citep{zhu2024phibe,zhu2025optimal}, the optimal-PhiBE provides a better approximation to the continuous-time optimal control problem compared to the Bellman equation when the underlying dynamics are smooth or the reward function oscillates, as it preserves the underlying SDE structure. Therefore, in this work, we adopt PhiBE to approximate the value function. The first-order Optimal-PhiBE corresponding to \citep{zhu2025optimal} is given by 
\begin{definition}[PhiBE]
\begin{equation}\label{optimal-phibe}
  \beta \hat{V}^*(s)
=
\sup_{a \in \mathcal{A}}
\left\{
    r(s,a)+
     \mathcal{L}_{\hat{b}, \hat{\Sigma}} \hat{V}^*(s)
\right\},  
\end{equation}
where
\begin{equation}\label{phibe-bsigma}
\begin{aligned}
\hat{b}(s,a)&=
\frac{1}{\Delta t}
\int (s' - s)\rho_\dt(s'|s,a)ds'\\
\hat{\Sigma}(s,a)
&=
\frac{1}{\Delta t}
\int(s'- s)(s' - s)^\top \rho_\dt(s'|s,a)ds'
\end{aligned}    
\end{equation}
\end{definition}
As shown in Theorem 3.5 of \citep{zhu2025optimal}, under suitable regularity conditions on the dynamics $b$ and $\sigma$, the optimal policy $\hat{\pi}^*(s)$ derived from PhiBE achieves an $O(\dt)$ approximation to the true optimal policy, when evaluated in terms of the true value function under the respective policies. We will use the PhiBE formulation in Section \ref{sec:off-policy-phibe} when only discrete-time information are available.

\section{PhiBE Q-Learning}\label{sec:off-policy-rl}

In this section, we aim to find a method that leverages the off-policy data. Off-policy data refers to experience collected under a behavior policy that is different from the target policy being optimized, in which the behavior policy may be unknown. Off-policy learning plays a crucial role in continuous-time reinforcement learning (CTRL). By allowing the use of data generated under arbitrary behavior policies, it significantly improves sample efficiency and makes it possible to reuse trajectories, which is especially important when collecting new data is costly. Furthermore, in many practical applications such as healthcare or autonomous driving, exploration is inherently risky. Off-policy learning provides a way to leverage existing datasets without requiring unsafe interactions. It also helps reduce the limitation caused by using data only from the current policy, enabling learning to reach an optimal policy. Finally, the ability to separate safe data collection from optimal policy optimization highlights the practical necessity of off-policy methods for CTRL, motivating our extension of PhiBE to this setting.

In standard RL, the key reason why $Q(s,a)$ naturally supports off-policy learning is because the Bellman equation 
\begin{equation}\label{discrete-time be}
 Q^\pi(s,a) = r(s,a) + \gamma \, \E_{s'\sim \rho(\cdot|s,a)} \l( \E_{a'\sim\pi(\cdot|s) }\l[ Q^\pi(s',a') \r]\r)   
\end{equation}
is determined by the unknown transition dynamics $\rho(\cdot|s,a)$ independent of the target policy $\pi$. While the Bellman equation for the value functions $V^\pi(s)$
\[
V^\pi(s) = r^\pi(s) + \gamma \, \E_{s'\sim \rho^\pi(\cdot|s)} \l( V^\pi(s') \r) 
\]  
is determined by the unknown marginal transition kernel  $\rho^\pi(s'|s) \;=\; \int \pi(a|s)\,\rho(s'|s,a)\, da$, which explicitly depends on $\pi$. As a result, when data are generated from a different behavior policy, the estimation of $V^\pi$ is inherently biased if the off-policy data is directly applied. In contrast, the Bellman optimality equation for $Q^*(s,a)$ is conditioned directly on $(s,a)$, and the transition kernel $\rho(s'|s,a)$ depends only on the environment dynamics, not on $\pi$. Thus, any sample $(s,a,s',r)$ collected under an arbitrary behavior policy provides an unbiased estimate of the right-hand side of the Bellman equation, making off-policy learning feasible for $Q$.  

Although the Bellman equation for $V^\pi(s)$ can, in principle, be adapted to off-policy learning by applying importance sampling to correct for the mismatch between the target policy $\pi$ and the behavior policy, this approach faces two challenges in practice. First, importance sampling often leads to high variance, especially in long-horizon problems or when the behavior and target policies differ substantially. Second, the behavior policy itself may be unknown, making it difficult to compute the required importance sampling ratios. These limitations motivate the use of $Q$-functions, whose Bellman equation is naturally off-policy without requiring such corrections.

A direct generalization of the discrete-time action-value function $Q^\pi(s,a)$ to the continuous-time setting is not feasible. As proved in \citep{jia2023q,tallec2019making}, the discrete-time action-value function $Q_{\Delta t}^\pi(s,a)$ defined as
\begin{equation}\label{def:q-dt}
Q_{\Delta t}^\pi(s,a) = 
\E_{\substack{
a_i \sim \pi(\cdot|s_i), i\geq1\\
s_{i+1} \sim \rho_{\Delta t}(\cdot \mid s_i, a_i), i\geq 0
}} \Big[\sum_{i = 0}^\infty\gamma^ir(s_i,a_i)\dt|s_0 = s, a_0 = a \Big]   , \quad \gamma = e^{-\beta \dt} 
\end{equation}
will converge to $V^\pi$ as $\Delta t \to 0$. The reason is that, as the discretization step $\Delta t \to 0$, the duration over which an action influences the environment also vanishes. In this limit, the effect of choosing action $a$ at state $s$ becomes negligible, and the action-value function $Q^\pi(s,a)$ effectively degenerates to the value function $V^\pi(s)$. 

In \citep{jia2023q, tallec2019making}, the authors propose an action-value function defined as
\begin{align*}
    q^\pi(s,a) := r(s,a) + [\mL V^\pi](s,a) - \beta V^\pi(s).
\end{align*}
This formulation, however, depends explicitly on $V^\pi$, whose estimation typically relies on on-policy transitions generated by $\pi$. We note that both \citet{jia2023q} and \citet{tallec2019making} propose off-policy algorithms for jointly estimating $q^\pi$ and $V^\pi$. Nevertheless, these approaches have two limitations when applied to discrete-time data. First, they require parameterizing two separate functions, $q^\pi$ and $V^\pi$, thereby introducing two sources of approximation error. Second, the resulting algorithms with discrete-time data do not fully exploit the smooth SDE structure. Here we instead seek a closed equation in a single unknown state--action function $Q$, using PhiBE to account for the finite observation interval and preserves the SDE structure even when only discrete-time data is available. 

Therefore, in order to handle the off-policy data within one value function and embed the SDE structure, we introduce a modified definition of the action-value function in the CTRL framework.

To learn from off-policy data using a single value function while incorporating the SDE structure, we introduce a modified continuous-time action-value function.
We first introduce the continuous-time action-value function in Section~\ref{sec:4.1.1}. Section~\ref{sec:off-policy-known} studies the known-dynamics setting, derives an iterative algorithm, and establishes its exponential convergence. We then use PhiBE to extend the analysis to unknown dynamics, where only the discrete-time transition kernel $\rho_{\Delta t}(s'|s,a)$ is available. Finally, we present a model-free, data-driven off-policy algorithm and evaluate it numerically in Section~\ref{sec:4.3}.

\subsection{Continuous-time state-action value function $Q(s,a)$}\label{sec:4.1.1}
In this section, we give the definition of the continuous-time state-action value function, and its property. 
\begin{definition}[Continuous-time Bellman equation for $Q^\pi(s,a)$]\label{def of Q-pi}
\begin{equation}\label{pde of q}
\beta Q^\pi(s,a) = r(s,a) + \mL \l[\E_{a\sim\pi} (Q^\pi)\r](s,a) 
\end{equation}

\end{definition}
Compare \eqref{pde of q} with the continuous-time Bellman equation for $V^\pi$ in \eqref{def of V-pi}:
\begin{equation}\label{eq:pe-V}
    \beta V^\pi(s) = r^\pi(s) + \mathcal{L}_{b^\pi,\Sigma^\pi} V^\pi(s), 
\quad b^\pi(s) = \E_{a\sim\pi}[b(s,a)], 
\quad \Sigma^\pi(s) = \E_{a\sim\pi}[\Sigma(s,a)].
\end{equation}
The two equations have a key structural difference. In \eqref{pde of q}, the differential operator is governed by the state-action dynamics $(b(s,a),\Sigma(s,a))$, which do not depend on the target policy. This parallels the discrete-time equation \eqref{discrete-time be} for $Q^\pi_\dt$ . By contrast, the operator for $V^\pi$ depends on the policy-averaged coefficients $(b^\pi,\Sigma^\pi)$ and hence on $\pi$. Thus, when the dynamics are unknown, the equation for $Q^\pi$ is naturally compatible with off-policy data, whereas the equation for $V^\pi$ is not.

The following theorem establishes the main properties of the continuous-time action-value function. We first define an admissible policy class. We call a stationary Markov policy $\pi$ admissible if $\pi\in C_{\mathrm{poly}}^{4,\alpha}
(\mathbb R^d;\mathcal A)$ or its density function $\pi\in C_{\mathrm{poly}}^{4,\alpha}
(\mathbb R^d\times\mathcal A)$.
The Theorems below are all based on the above policy class.

\begin{theorem}\label{thm:q-pi-property}
Suppose that Assumption~\ref{ass} holds, and let
$\pi\in\Pi_{\mathrm{ad}}$. Then the PDE \eqref{pde of q} admits a unique solution $Q^\pi$ among functions $Q$ for which $\mathbb{E}_{a\sim\pi}[Q^\pi(s,a)] \in
C_b(\mathbb R^d)
\cap
C_{\mathrm{loc}}^2(\mathbb R^d).$
\begin{enumerate}
\item Its expectation under the current policy equals $V^\pi$:
\[
\E_{a\sim\pi(\cdot\mid s)}[Q^\pi(s,a)] = V^\pi(s),
\]
where $V^\pi$ is defined in \eqref{def of V-pi}.

\item It has the equivalent representation
\[
Q^\pi(s,a) = \lim_{\Delta t\to 0} \left[\frac{Q^\pi_{\Delta t}(s,a) - V^\pi_{\Delta t}(s)}{1-\gamma} + V_\dt^\pi(s)\right]
\]
where $Q^\pi_{\Delta t}$ is defined in \eqref{def:q-dt} and $V^\pi_\dt(s)=\E_{a\sim\pi(\cdot\mid s)}[Q^\pi_\dt(s,a)]$.
\item If the greedy policy $\pi'(s)\in\argmax_a Q^\pi(s,a)$ is still in the admissible policy class, then there is a policy improvement:
\[
V^{\pi'}(s) \geq V^\pi(s), \quad \forall s.
\]
\end{enumerate}
\end{theorem}

The proof is given in Appendix~\ref{sec:proof of thm-q-pi-property}.

Several remarks regarding the continuous-time $Q^\pi(s,a)$ are followed. 

First, the first property yields the equivalent equation
\begin{equation}\label{eq-Q-2}
\beta Q^\pi(s,a) = r(s,a) +  \mL [V^\pi](s,a).
\end{equation}
We use Definition~\ref{def of Q-pi} rather than \eqref{eq-Q-2} because it is expressed entirely in terms of $Q^\pi$ and avoids explicitly computing $V^\pi$, a feature that is essential for off-policy learning.

Second, property~2 has a natural physical
interpretation. Starting from state $s$, one may view the initial action $a$ as a short
control pulse applied for a time interval of length $\dt$, after which the system follows
the policy $\pi$. The difference $Q^\pi_\dt(s,a)-V^\pi_\dt(s)$ measures the incremental
effect of this pulse relative to drawing the initial action from $\pi$. This effect
vanishes linearly with $\dt$ because the pulse acts for an increasingly short time \cite{}.
This effect
vanishes linearly with $\dt$ because the pulse acts for an increasingly short time.
To interpret the normalization, recall that the continuous-time discount kernel is
$e^{-\beta t}$. The fraction of the total discounted horizon assigned to the first
interval $[0,\dt]$ is
\[
\frac{\int_0^\dt e^{-\beta t}\,dt}
     {\int_0^\infty e^{-\beta t}\,dt}
=1-e^{-\beta\dt}=1-\gamma.
\]
Equivalently, $1-\gamma$ is the probability mass assigned to $[0,\dt]$ by the
normalized discount density $\beta e^{-\beta t}$. Hence, $\frac{Q^\pi_\dt(s,a)-V^\pi_\dt(s)}{1-\gamma}$
rescales the vanishing pulse effect relative to the full discounted horizon, and can be
interpreted as the finite, horizon-adjusted effect of applying action $a$ over the short
interval. Adding $V^\pi_\dt(s)$ then restores the baseline continuation value of
following $\pi$. Thus, $\frac{Q^\pi_\dt(s,a)-V^\pi_\dt(s)}{1-\gamma}+V^\pi_\dt(s)$
represents the baseline value of state $s$ plus a correction that captures the
physical effect of applying action $a$ instantaneously.

We next connect this limit to the instantaneous advantage function. As shown in \citep{jia2023q}: $\lim_{\dt\to0}\frac{Q^\pi_\dt - V^\pi_\dt}{\dt} = q^\pi(s,a) = r(s,a) + \mL [V^\pi](s,a) - \beta V^\pi(s)$. Therefore, by property 1), $Q^\pi$ can be equivalently written as
\begin{equation}\label{def of q}
Q^\pi(s,a) = \frac1\beta q^\pi(s,a) + V^\pi(s) 
\end{equation}
Equation~\eqref{def of q} also has a physical interpretation. The quantity
$V^\pi(s)$ is the baseline value obtained by following policy $\pi$, whereas
$q^\pi(s,a)$, also known as instantaneous advantage function, is the instantaneous excess value rate generated by initially applying
action $a$. 
However, $q^\pi(s,a)$ cannot be added
directly to $V^\pi(s)$ because the two quantities have different units:
$q^\pi(s,a)$ is a rate, with units of value per unit time, whereas $V^\pi(s)$ has units
of value. Because the discount rate $\beta$ has units of inverse time, dividing
$q^\pi(s,a)$ by $\beta$ converts this instantaneous rate into a quantity with units
of value. Thus, $q^\pi(s,a)/\beta$ represents the value-level correction associated
with initially applying action $a$, relative to the baseline $V^\pi(s)$.  Moreover, a positive $q^\pi(s,a)$ indicates that action $a$ is locally better
than the policy baseline and yields $Q^\pi(s,a)>V^\pi(s)$, while a negative
$q^\pi(s,a)$ indicates the opposite. This is why the greedy policy $\pi'(s) = \argmax_a q(s,a) = \argmax_a Q(s,a)$ gives policy improvement.

Finally, consider the entropy-regularized control problem with
$
r_\lambda^\pi(s)
=
\E_{a\sim\pi(\cdot\mid s)}
\left[r(s,a)-\lambda\log\pi(a\mid s)\right],.
$
The corresponding soft policy-improvement step is $\pi'(a\mid s)
\propto
\exp\!\left(q_\lambda^\pi(s,a)/\lambda\right).$
Provided the normalizing integral is finite and $\pi'$ is
admissible, this softmax update satisfies
$V_\lambda^{\pi'}(s)\geq V_\lambda^\pi(s)$ under the
regularity and integrability assumptions of
\citet{jia2023q}.

Next, we define the continuous-time optimal Bellman equation for $Q^*(s,a)$ as follows.
\begin{definition}[Continuous-time optimal Bellman equation for $Q^*(s,a)$]\label{def of Q-pde}
\begin{equation}\label{pde of Qstar}
    \beta Q^*(s,a) = r(s,a) + \mL \big[\max_{a'} Q^*(s,a')\big] .
\end{equation}
\end{definition}

\begin{theorem}[Consistency of the optimal $Q$-Bellman equation]
\label{thm:q-star-property}
Suppose that Assumption~\ref{ass} holds. Then the PDE \eqref{pde of Qstar} admits a unique solution $Q^*$ among functions $Q$ for which $\max_{a\in\mathcal A}Q^*(s,a) \in
C_b(\mathbb R^d)
\cap
C_{\mathrm{loc}}^2(\mathbb R^d).$
\begin{enumerate}
\item
$\max_{a\in\mathcal A}Q^*(s,a)
=
V^*(s)$, where $V^*$ is the unique solution to the HJB equation \eqref{eq:hjb} in $C_b(\mathbb R^d)
\cap
C_{\mathrm{loc}}^2(\mathbb R^d).$

\item
If a measurable selector
\[
\pi^*(s)\in\argmax_{a\in\mathcal A}Q^*(s,a)
\]
belongs to $\Pi_{\mathrm{ad}}$, then $\pi^*$ is an optimal policy and
\[
V^{\pi^*}(s)
=
V^*(s)
=
\max_{\pi\in\Pi_{\mathrm{ad}}}V^\pi(s),
\qquad
Q^{\pi^*}(s,a)=Q^*(s,a).
\]
\end{enumerate}
\end{theorem}
The proof is given in Appendix~\ref{sec:proof of thm-q-star-property}. Several remark regarding the above theorem are in order. Theorem~\ref{thm:q-star-property} establishes the consistency between the proposed
continuous-time $Q$-function and the standard optimal value function. The first property shows
that optimizing the initial action of $Q^*(s,a)$ recovers the optimal state value $V^*(s)$. The second property shows the greedy policy from $Q^*(s,a)$ is an optimal policy as long as it is admissible. Thus, the continuous-time $Q^*$ preserves the standard relationships among the optimal action-value function,
the optimal value function, and the optimal policy in discrete-time RL.

\subsection{Q-learning under Known Dynamics}\label{sec:off-policy-known}
In this section, we introduce an iterative algorithm to approximate $Q^*(s,a)$, the solution to \eqref{pde of Qstar}, under linear bases. First, to simplify the notation, we define the nonlinear optimality operator
\begin{equation}\label{eq:T-operator}
    \text{Continuous-time Bellman-Q operator:} \quad \mathcal TQ(s,a)
    :=\mathcal L_{b,\Sigma}\l(\max_{a'\in\mathcal A}Q(s,a')\r).
\end{equation}
Then the continuous-time optimal $Q$-Bellman equation
\eqref{pde of Qstar} can be equivalently written as
\begin{equation}\label{eq:Q-residual-equation}
    r+\mathcal TQ^*-\beta Q^*=0.
\end{equation}
We first introduce an artificial algorithmic time $\tau$ and evolve an
action-value function in the direction of its Bellman residual,
\begin{equation}\label{eq:Q-residual-flow}
    \frac{\partial}{\partial \tau}Q(\tau,s,a)
    =r(s,a)+\mathcal TQ(\tau)(s,a)-\beta Q(\tau,s,a).
\end{equation}
Note that every stationary point of the above PDE solves
\eqref{eq:Q-residual-equation}. The above equation \eqref{eq:Q-residual-flow} is
used only to derive the numerical iteration. We hope that one can solve the above equation numerically, and the evolution will lead to an approximation of the stationary solution.

Let $\mu$ be a probability measure on $\mathcal S\times\mathcal A$ and
write $\langle\cdot,\cdot\rangle_\mu$ and $\|\cdot\|_\mu$ for the inner
product and norm in $L^2(\mu)$. Here $\mu$ represents the p.d.f. of the off-policy data. Let $\Phi(s,a):=(\phi_1,\ldots,\phi_d)^\top$ be a linear bases, and we seek to approximate \eqref{eq:Q-residual-flow} in the linear space $\mathcal V:=\{\Phi(s,a)^\top\theta:\theta\in\mathbb R^d\}\subset L^2(\mu)$. Define the Gram matrix 
\begin{equation}\label{eq:gram-matrix}
G:=\langle\Phi,\Phi^\top\rangle_\mu
      =\int_{\mathcal S\times\mathcal A}
        \Phi(s,a)\Phi(s,a)^\top\mu(ds,da).
\end{equation}
When $G\succ0$, for any $f\in L^2(\mu)$, define its $L^2(\mu)$-orthogonal projection onto $\mathcal V$ by 
\begin{equation}\label{def of proj}
    \Pi_\Phi f:=\arg\min_{g\in\mathcal V}\|f-g\|_\mu^2 = \Phi^\top G^{-1}\langle f,\Phi\rangle_\mu
\end{equation}
The explicit expression on the right-hand side of \eqref{def of proj} is obtained by solving the quadratic minimization problem in $\theta$ as every $g\in\mathcal V$ can be written as $g=\Phi^\top\theta$.

One way to do find the approximation solution in \eqref{eq:Q-residual-flow} in the linear bases is the Galerkin method,
\begin{equation}\label{eq:galerkin-flow}
\la \frac{\partial}{\partial \tau}Q_{\th(\tau)}, \Phi\ra_\mu
    =\la r(s,a)+\mathcal TQ_{\th(\tau)}-\beta Q_{\th(\tau)}, \Phi\ra_\mu.
\end{equation}
When $G$ is invertible, \eqref{eq:galerkin-flow} is equivalent to
\begin{equation}\label{eq:preconditioned-ode}
    \dot\theta(t)=G^{-1}
    \left\langle r+\mathcal TQ_{\theta(t)}-\beta Q_{\theta(t)},
    \Phi\right\rangle_\mu.
\end{equation}
and its stationary point $\theta^*$ satisfies the Galerkin approximation $\widehat Q^*:=Q_{\theta^*}$ to the continuous-time optimal Q equation \eqref{pde of Qstar}, 
\begin{equation}\label{galerkin-solu}
     \left\langle r+\mathcal TQ_{\theta^*}-\beta Q_{\theta^*},
    \Phi\right\rangle_\mu=0.
\end{equation}
By the definition of the projection in \eqref{def of proj}, the Galerkin equation \eqref{galerkin-solu} can be equivalently written as,
\begin{equation}\label{eq:fix-pt-proj}
    Q_{\th^*} = \frac{1}{\beta}\Pi_\Phi(r+ \mathcal TQ_{\theta^*}).
\end{equation}
Applying forward Euler to \eqref{eq:preconditioned-ode} gives the iterative continous-time Q-learning algorithm under the linear bases:
\begin{equation}\label{eq:newalgoite}
\text{Continuous-time Q-learning:}\quad \theta_{n+1}=\theta_n+\alpha_nG^{-1}
    \left\langle r+\mathcal TQ_{\theta_n}-\beta Q_{\theta_n},
    \Phi\right\rangle_\mu .
\end{equation}
with $G$ the Gram matrix defined in \eqref{eq:gram-matrix} and $\mathcal T$ the continuous-time Bellman-Q operator \eqref{eq:T-operator}. By the definition of the projection in \eqref{def of proj}, one can write the $Q$-update equivalently as,
\begin{equation}\label{update-proj}
    Q_{\theta_{n+1}}=Q_{\theta_n}+\alpha_n\Pi_\Phi (r+\mathcal TQ_{\theta(t)}-\beta Q_{\theta(t)})
\end{equation}

Regarding the continuous-time Q-learning when the underlying dynamics are known, several remarks follows in order.  For $\theta\in\mathbb R^d$, define
\begin{equation}\label{eq:maximized-basis-function}
    V_\theta(s):=\max_{a\in\mathcal A}\Phi(s,a)^\top\theta.
\end{equation}
To run the strong-form iteration \eqref{eq:newalgoite}, the maximum must be
finite and attained at every iterate, and $V_{\theta_n}$ must be twice
continuously differentiable in the state with
$\mathcal L_{b,\Sigma}V_{\theta_n}\in L^2(\mu)$.  For the general result below, the approximation space is
understood to be chosen so that these requirements hold for every
$Q_\theta\in\mathcal V$. For compact $\mathcal A$, continuity in $a$ guarantees that the maximum is
attained. A sufficient condition for $V_\theta\in C^2$ is that $\Phi$ is
$C^2$ jointly in $(s,a)$ and that the maximizing action is unique, interior,
and has a nonsingular negative-definite action Hessian. One can avoid this restriction by using the softmax alternative in Remark \ref{rmk:softmax}.

\begin{remark}[A softmax alternative]\label{rmk:softmax}
A smooth alternative can be constructed as follows. For $\lambda>0$, define the softmax policy
\begin{equation}\label{eq:softmax-policy-known}
 \pi_{\theta,\lambda}(da\mid s)
 \propto \exp\{\Phi(s,a)^\top\theta/\lambda\}\,\nu(da),
\end{equation}
where $\nu$ is a fixed reference measure. Replace
\eqref{eq:maximized-basis-function} by the entropy-regularized
soft maximum
\begin{equation}\label{eq:softmax-expected-value}
 V_{\theta,\lambda}(s)
 :=\lambda\log\int_{\mathcal A}
 \exp\{\Phi(s,a)^\top\theta/\lambda\}\,\nu(da).
\end{equation}
Provided the integral is finite and differentiation under the
integral is valid through second order, this replacement avoids
nonsmoothness caused by switching maximizers. It yields an
entropy-regularized formulation \citep{wang2020reinforcement},
with entropy weight $\beta\lambda$ under our scaling.
A fixed positive temperature generally introduces regularization
bias. Suitable temperature annealing can remove this bias,
as established for policy mirror descent by
\citet{sethi2025entropy}. The convergence results below apply
only to the hard-maximum iteration. The smoothed and annealed
variants require separate analysis.
\end{remark}

We will first prove that the Galerkin solution $\hat{Q}^* = \Phi^\top\th^*$ to \eqref{galerkin-solu} is a good approximation to the true optimal $Q^*$ in Theorem \ref{thm:galerkin}, and then in Theorem \ref{thm:galerkin_convergence} that the above algorithm will converge to the Galerkin solution.
\paragraph{Weighted adjoint and stability constants.}
We will study the theoretical guarantee of the proposed algorithm in the following weighted norm $ \|Q\|_\mu^2
 :=\int_{\mathcal S\times\mathcal A}|Q(s)|^2\mu(s,a)dsda.$ where $\mu(s,a)$ is a probability density function. We further define a weighted adjoint operator 
\begin{equation}\label{eq:weighted-adjoint}
 \mathcal L_{b,\Sigma}^{*,\mu}g
 :=\l\{\begin{aligned}
     &\frac{1}{\mu}\l[-\nabla_s\!\cdot(bg\mu)
  +\frac12\sum_{j,k}\partial_{s_js_k}(\Sigma_{jk}g\mu)\r],\quad (s,a)\in \{\mu>0\}\\
  & 0, \quad (s,a)\not\in \{\mu>0\}
 \end{aligned}\r.,
\end{equation}
and stability constants
\begin{equation}\label{eq:max-stability-constant}
 C_{\max}
 :=\sup_{\theta\ne\eta}
 \frac{\|V_\theta-V_\eta\|_\mu}
      {\|\Phi^\top(\theta-\eta)\|_\mu},
 \qquad
 C_{\mathcal L^*}
 :=\sup_{\|\th\|_2=1}
 \frac{\|\mathcal L_{b,\Sigma}^{*,\mu}(\Phi^\top\th)\|_\mu}
      {\|\Phi^\top \th\|_\mu}.
\end{equation}
The second quotient is homogeneous of degree zero, so the unit-sphere
restriction is equivalent to taking the supremum over $\th\ne0$. Set
\begin{equation}\label{eq:ibp-constant}
 C_{\mathcal T}:=C_{\max}C_{\mathcal L^*}.
\end{equation}

\begin{assumption}[Galerkin stability]\label{ass-conv}
\begin{itemize}
 \item[(A1)] The basis functions
 $\phi_1,\ldots,\phi_d$ belong to $L^2(\mu)$ and are linearly independent
 in $L^2(\mu)$, or equivalently, $G$ is positive definite.

 \item[(A2)] For every $Q_\theta\in\mathcal V$ and every
 $g\in\mathcal V$, the weak adjoint identity
 \begin{equation}\label{eq:weak-adjoint-identity}
  \langle\mathcal L_{b,\Sigma}V_\theta,g\rangle_\mu
  =\langle V_\theta,\mathcal L_{b,\Sigma}^{*,\mu}g\rangle_\mu
 \end{equation}
 holds. Moreover, $C_{\max}$ and $C_{\mathcal L^*}$ are finite and
 \begin{equation}\label{def of cbeta}
  \beta>C_{\mathcal T}=C_{\max}C_{\mathcal L^*}.
 \end{equation}
\end{itemize}
\end{assumption}

\begin{remark}[Interpretation of the stability constants]
\label{rem:constant-interpretation}
The constant $C_{\max}$ is the amplification produced by greedy action
selection: it compares an average $Q$-function perturbation under $\mu$ with
the perturbation after maximizing over actions. A large $C_{\max}$ means that a perturbation that is small in $Q_\th$ can produce a much larger change after maximization over actions. Such amplification can arise from two distinct mechanisms. First, the distribution $\mu$ may assign little mass to actions that are optimal. A perturbation concentrated near such actions may then have a small $L^2(\mu)$ norm while substantially changing the maximized value. Second, the maximization problem itself may be ill-conditioned, so that its value is highly sensitive to perturbations of $Q$. In the quadratic LQR example, this occurs when $H_{aa}$ approaches singularity. Conversely, $C_{\max}$ is better controlled when potentially optimal actions are sufficiently covered by $\mu$ and the maximization is uniformly well-conditioned.

The constant $C_{\mathcal L^*}$ is the operator norm of the weighted adjoint restricted to the Galerkin space.
It therefore measures the largest relative $L^2(\mu)$ amplification produced by the weighted adjoint over all directions in $\mathcal V$. It is a property of the interaction among the dynamics $(b,\Sigma)$, the weighting density $\mu$, and the chosen approximation space. Strong or rapidly varying coefficients, rapid variation in $\mu$, or highly oscillatory bases can make this constant large because the weighted adjoint contains first- and second-order derivatives.

Consequently, $ C_{\mathcal T}$ measures the combined amplification arising from greedy maximization and the subsequent adjoint operation in the Galerkin estimate. The condition $\beta>C_{\mathcal T}$ requires discounting to dominate this combined amplification. 
\end{remark}

\begin{remark}[Sufficient conditions]\label{rem:primitive-conditions}
Assumption (A2) uses integration by parts only in the convergence proof.
The algorithm \eqref{eq:newalgoite} is unchanged. The weak adjoint identity
\eqref{eq:weak-adjoint-identity} holds, for example, on a periodic state
domain, on $\mathbb R^n$ when the boundary terms decay at infinity, or
under compatible zero-flux boundary conditions. 

The constants can be verified in the finite-dimensional Galerkin space. Let
\[
 \lambda_1:=\lambda_{\min}(G),\qquad
 \Phi_*:=\mathcal L_{b,\Sigma}^{*,\mu}\Phi,\qquad
 G_*:=\langle\Phi_*,\Phi_*^\top\rangle_\mu.†
\]
Therefore, one sufficient condition for the constant $C_{\mathcal L^*}$, $ C_{\max}$ being finite is when every component of
$\Phi_*$ belongs to $L^2(\mu)$, and $\lambda_1 > 0,  M_\Phi:=\sup_{(s,a)\in\mathcal S\times\mathcal A}\|\Phi(s,a)\|_2<\infty$.
Then \[
C_{\mathcal L^*}^2
 =\lambda_{\max}\!\left(G^{-1/2}G_*G^{-1/2}\right)
 \leq\frac{\lambda_{\max}(G_*)}{\lambda_1}, \quad C_{\max}\leq
 \frac{M_\Phi}{\sqrt{\lambda_1}}.
 \] 
Note that the uniform bound involving $M_\Phi$ is only a convenient and often
conservative sufficient condition.
\end{remark}

\begin{lemma}\label{lemma:projected-operator-Lipschitz}
Under Assumption~\ref{ass-conv}, for every $Q_\th,Q_\eta\in\mathcal V$,
\begin{equation}\label{eq:projected-Lipschitz}
\|\Pi_\Phi(\mathcal TQ_\th-\mathcal TQ_\eta)\|_\mu
\leq C_{\mathcal T}\|Q_\th-Q_\eta\|_\mu.
\end{equation}
\end{lemma}
\begin{proof}
By the definition of the projection in \eqref{def of proj}, we know $\Pi_\Phi(\mathcal TQ_\th-\mathcal TQ_\eta) \in \mathcal{V}$. In addition, for $\forall h\in\mathcal V$
\[
\|h\|_\mu
=
\sup_{\substack{g\in\mathcal V\\ \|g\|_\mu=1}}
|\langle h,g\rangle_\mu|.
\]
This follows because the Cauchy--Schwarz inequality gives ``$\leq$,'' while equality is attained by choosing $g=h/\|h\|_\mu$. By \eqref{eq:T-operator},
$\mathcal TQ_\th-\mathcal TQ_\eta=\mathcal L_{b,\Sigma} (V_\theta-V_\eta)$. Therefore, one has
\begin{align*}
 &\|\Pi_\Phi(\mathcal TQ_\th-\mathcal TQ_\eta)\|_\mu  =
 \sup_{\substack{g\in\mathcal V\\ \|g\|_\mu=1}}
 \left|\langle \mathcal L_{b,\Sigma} (V_\theta-V_\eta),  g\rangle_\mu\right| \\
 \leq&
 \| V_\theta-V_\eta\|_\mu
 \sup_{\substack{g\in\mathcal V\\ \|g\|_\mu=1}}
 \|\mathcal L_{b,\Sigma}^{*,\mu}g\|_\mu \leq C_{\mathcal L^*}C_{\max}
 \|Q-\th - Q_\eta\|_\mu
\end{align*}
where the last inequality follow from
\eqref{eq:max-stability-constant} and \eqref{eq:ibp-constant}. The proof is completed by the definition of $C_{\mathcal{L}}$ in \eqref{eq:ibp-constant}, 
\end{proof}

\begin{lemma}\label{lemma:projected-bellman-contraction}
Under Assumption~\ref{ass-conv}, the projected Bellman map $\mathcal B(Q):=\frac1\beta\Pi_\Phi(r+\mathcal TQ)$ is a contraction on $\mathcal V$ with contraction constant
$C_{\mathcal T}/\beta<1$. That is, 
\[
\text{for }\forall Q, P\in \mathcal{V}, \quad \|\mathcal B(Q)-\mathcal B(P)\|_\mu
\leq\frac{C_{\mathcal T}}\beta\|Q-P\|_\mu.
\]
\end{lemma}
\begin{proof}
For  $\forall Q,P\in\mathcal V$, by the linearity of $\Pi_\Phi$,
\begin{align*}
 \mathcal B(Q)-\mathcal B(P)
 &=\frac1\beta\left[
   \Pi_\Phi(r+\mathcal TQ)-\Pi_\Phi(r+\mathcal TP)\right] =\frac1\beta\Pi_\Phi(\mathcal TQ-\mathcal TP).
\end{align*}
Taking the $L^2(\mu)$ norm and applying
Lemma~\ref{lemma:projected-operator-Lipschitz} gives
\[
\|\mathcal B(Q)-\mathcal B(P)\|_\mu
=\frac1\beta
 \|\Pi_\Phi(\mathcal TQ-\mathcal TP)\|_\mu
\leq\frac{C_{\mathcal T}}\beta\|Q-P\|_\mu.
\]
Assumption~\ref{ass-conv}/(A2) states that
$\beta>C_{\mathcal T}$, and therefore $0\leq\frac{C_{\mathcal T}}\beta<1$.
\end{proof}

\begin{theorem}[Galerkin solution accuracy]\label{thm:galerkin}
Suppose Assumption~\ref{ass-conv} holds.  Then the projected equation
\eqref{galerkin-solu} has a unique solution $\widehat Q^*\in
\mathcal V$. Suppose additionally that the exact solution $Q^*$ belongs to
$L^2(\mu)$, and denote $\overline Q:=\Pi_\Phi Q^*$ as its orthogonal projection,
Then
\begin{equation}\label{eq:improved-galerkin-bound}
    \|\widehat Q^*-Q^*\|_\mu
    \leq
    \inf_{Q\in\mathcal V}\|Q-Q^*\|_\mu
    +\frac{1}
           {\beta-C_{\mathcal T}}\|  \Pi_\Phi(r+\mathcal T \bar{Q}^* - \beta \bar{Q}^*  )\|_\mu.
\end{equation}
Consequently, if $Q^*\in\mathcal V$, then $\widehat Q^*=Q^*$.
\end{theorem}

Several remarks concerning the above theorem are in order. First, $\overline{Q}:=\Pi_\Phi Q^*$ is the best approximation of $Q^*$ in the Galerkin space $\mathcal V$ under the $L^2(\mu)$ norm. Hence, the first term on the RHS of \eqref{eq:improved-galerkin-bound} represents the approximation, or model-misspecification, error induced by the choice of $\mathcal V$. Second, $\varepsilon_{\mathcal V}$ measures how strongly this unresolved approximation error feeds back into the Bellman residual. It can therefore be interpreted as a projected Bellman inconsistency. Its contribution to the total error is amplified by the stability factor $O\bigl((\beta-C_{\mathcal T})^{-1}\bigr)$. This factor plays a role analogous to the familiar $O\bigl((1-\gamma)^{-1}\bigr)$ amplification factor in discrete-time reinforcement-learning error bounds.

\begin{proof}
By Lemma~\ref{lemma:projected-bellman-contraction} and the Banach
fixed-point theorem \cite{}, $\mathcal B$ has a unique fixed point
$\widehat Q^*\in\mathcal V$.  By \eqref{eq:fix-pt-proj}, this
fixed point is exactly the Galerkin solution.

Projecting the equation for $Q^*$ \eqref{eq:T-operator} gives,
\begin{equation}\label{eq:eq-1}
    0 = \Pi_\Phi (r + \mathcal T Q^* - \beta Q^*) = \Pi_\Phi (r + \mathcal T Q^*  - \beta Q^*)
\end{equation}
Subtracting the Galerkin equation \eqref{galerkin-solu} and using the linearity
of $\Pi_\Phi$ gives
\[
\beta\Pi_\Phi (Q^*-\widehat Q^*)  = \Pi_\Phi (\mathcal T Q^* - \mathcal T \widehat Q^* ) 
\]
Taking norms and using the fact that $\Pi_\Phi (Q^* - \widehat Q^*) = \bar{Q}^* -  \widehat Q^*$ and $\bar{Q}^*,\widehat Q^*\in \mathcal{V} $ gives
\[
\begin{aligned}
 &\beta\|\widehat Q^*-\overline Q\|_\mu  =\| \Pi_\Phi (\mathcal T Q^* - \mathcal T \widehat Q^*)\|_\mu\leq \| \Pi_\Phi (\mathcal T Q^* -\mathcal T \bar{Q}^*)\|_\mu + \|\Pi_\Phi(\mathcal T \bar{Q}^* - \mathcal T \widehat Q^*)\|_\mu \\
 \leq &\| \Pi_\Phi( \beta Q^* - r -\mathcal T \bar{Q}^*)\|_\mu + C_{\mathcal{T}}  \|\widehat Q^*-  \bar{Q}^*\|_\mu    = \|  \Pi_\Phi(r+\mathcal T \bar{Q}^* - \beta \bar{Q}^*  )\|_\mu + C_{\mathcal{T}}  \|\widehat Q^*-  \bar{Q}^*\|_\mu 
\end{aligned}
\]
where the last inequality comes from applying
Lemma~\ref{lemma:projected-operator-Lipschitz} and using the equality \eqref{eq:eq-1}. Since $\beta>C_{\mathcal T}$, this is equivalent to
\[
 \|\widehat Q^*-\overline Q\|_\mu
 \leq\frac{\varepsilon_{\mathcal V}}{\beta-C_{\mathcal T}}.
\]
The triangle inequality and the best-approximation property
$\|\overline Q-Q^*\|_\mu
=\inf_{Q\in\mathcal V}\|Q-Q^*\|_\mu$ prove
\eqref{eq:improved-galerkin-bound}. If $Q^*\in\mathcal V$, then
$\overline Q=Q^*$ and \eqref{eq:Q-residual-equation} gives
$\varepsilon_{\mathcal V}=0$.
\end{proof}

\begin{theorem}[Convergence of Q-learning under known dynamics]
\label{thm:galerkin_convergence}
Suppose Assumption~\ref{ass-conv} holds and let the learning rate be
constant.  If
\begin{equation}\label{eq:improved-stepsize-range}
    0<\alpha<\frac{2}{\beta+C_{\mathcal T}},
\end{equation}
then the iterates in \eqref{eq:newalgoite} converge linearly to the unique
Galerkin coefficient $\theta^*$ defined by \eqref{galerkin-solu}.  In
particular,
\begin{equation}\label{eq:improved-linear-rate}
    \|\theta_n-\theta^*\|_G
    =\|Q_n-\widehat Q^*\|_\mu
    \leq (|1-\alpha\beta|+\alpha C_{\mathcal T})^n\|Q_0-\widehat Q^*\|_\mu,
\end{equation}
The upper bound of the convergence rate is minimized at $\alpha=1/\beta$, for which
\begin{equation}\label{eq:optimal-improved-rate}
    \|Q_n-\widehat Q^*\|_\mu
    \leq\left(\frac{C_{\mathcal T}}\beta\right)^n
    \|Q_0-\widehat Q^*\|_\mu.
\end{equation}
\end{theorem}

The decomposition of the residual into
$\Pi_\Phi\mathcal T-\beta I$ gives the explicit admissible interval
\eqref{eq:improved-stepsize-range}. At $\alpha=1/\beta$, the contribution
of the identity term vanishes and the contraction factor is
$C_{\mathcal T}/\beta$. 

\begin{proof}
Subtracting the Galerkin solution $\alpha\Pi_\Phi(r + \mathcal{T}\widehat Q^* - \beta \widehat Q^*) = 0$
 from \eqref{update-proj} yields
 \[
\begin{aligned}
 Q_{\th_{n+1}}
 =&Q_{\th_n}
 +\alpha\Pi_\Phi(\mathcal TQ_{\th_n}  - \mathcal{T}\widehat Q^*) -\beta \alpha(Q_{\th_n} -  \widehat Q^* ) 
\end{aligned}
\]
Substracting $\widehat Q^*$ on both sides, and taking the $L^2(\mu)$ norm and applying
\eqref{eq:projected-Lipschitz} gives
\[
 \|Q_{n+1}-\widehat Q^*\|_\mu
 \leq\l(|1-\alpha\beta|+\alpha C_{\mathcal T}\r)
       \|Q_n-\widehat Q^*\|_\mu.
\]
For $0<\alpha\leq1/\beta$,
$q(\alpha)=1-\alpha(\beta-C_{\mathcal T})<1$.  For
$\alpha\geq1/\beta$, $q(\alpha)=\alpha(\beta+C_{\mathcal T})-1<1$
exactly when $\alpha<2/(\beta+C_{\mathcal T})$.  Hence
\eqref{eq:improved-stepsize-range} implies contraction.  The first branch
is decreasing and the second branch is increasing, so the minimum occurs
at $\alpha=1/\beta$ and equals $C_{\mathcal T}/\beta$.  Finally,
\[
 \|Q_n-\widehat Q^*\|_\mu^2
 =(\theta_n-\theta^*)^\top G(\theta_n-\theta^*)
 =\|\theta_n-\theta^*\|_G^2,
\]
which proves the claimed convergence to the Galerkin solution.
\end{proof}

\subsection{Q-learning under unknown dynamics}\label{sec:off-policy-phibe}
In the previous section, we assumed that the continuous dynamics $b, \sigma$ are known. In this section, we assume $b, \sigma$ are unknown. In Section \ref{sec:4.2.1}, we introduce PhiBE to naturally apply the algorithm and theoretical guarantees when only the discrete-time transition dynamics are known. In \ref{sec:4.2.2}, we describe how to implement the update rule \eqref{eq:newalgoite} in a data-driven manner when the transition dynamics is unknown and only the state-action-next-state-reward tuples  $\{(s_i, a_i,s'_i,r_i)\}_{i=1}^n$ as defined in \eqref{data} are available.

\subsubsection{PhiBE setting}\label{sec:4.2.1}
When only the discrete-time transition dynamics $\rho_{\Delta t}(s'|s, a)$ are available, one can replace the unknown drift and diffusion $b,\Sigma$ in the equation \eqref{pde of Qstar} for $Q^*$ with the approximated $\hat{b}$ and $\hat{\Sigma}$ according to \eqref{phibe-bsigma}, and one ends up with the Optimal-PhiBE-Q equation

\begin{definition}[Optimal-PhiBE-Q]\label{def:phibe-q}
\begin{equation}\label{pde of Qhatstar}
    \beta \hat{Q}^*(s,a) = r(s,a) + \mathcal{L}_{\hat{b},\hat{\Sigma}} \big[\max_{a} \hat{Q}^*(s,a)\big] , 
\end{equation}
where $\hat{b}, \hat{\Sigma}$ are defined in \eqref{phibe-bsigma}. 
\end{definition}

Substituting $(\hat{b}, \hat{\Sigma})$ into \eqref{pde of Qstar} yields a corresponding operator, whose solution corresponds to the optimal PhiBE solution to \eqref{optimal-phibe}.
Solving the Optimal-PhiBE-Q equation in Definition \ref{def:phibe-q}, the same algorithm \eqref{eq:newalgoite} applies with the approximated $\hat{b}, \hat{\Sigma}$. Therefore as long as $\hat{b}, \hat{\Sigma}$ satisfies the Assumption (A2), the corresponding algorithm 
\begin{equation}\label{eq:newalgoite-phibe}
\begin{aligned}
\theta_{n+1} =& \theta_n  + \alpha_n G^{-1}\bigg\langle  r(s,a) + \mathcal{L}_{\hat{b},\hat{\Sigma}}  \Big(\max_a [\Phi(s,a)^\top \theta_n]\Big)-\beta \Phi(s,a)^\top \theta_n, \ \Phi(s,a) \bigg\rangle 
\end{aligned}
\end{equation}
recovers the same limiting behavior as established in Theorem \ref{thm:galerkin} and Theorem \ref{thm:galerkin_convergence}.

\subsubsection{Data-driven algorithm}
\label{sec:4.2.2}

When only discrete-time off-policy data \eqref{data} is available, our
goal is to approximate the optimal action-value function $Q^*(s,a)$.

Following the algorithm \eqref{eq:newalgoite-phibe} introduced in the
last section, we approximate $Q^*(s,a)$ using the basis
$\Phi(s,a)=(\phi_1,\dots,\phi_n)^\top$. {We then construct the empirical Gram matrix from the complete
off-policy data set $\mathcal B$ as}
\begin{equation}
{
\widehat G
:=
\frac{1}{|\mathcal B|}
\sum_{(s_i,s_i',a_i,r_i)\in\mathcal B}
\Phi(s_i,a_i)\Phi(s_i,a_i)^\top .
}
\label{eq:empirical-data-gram}
\end{equation}

In each iteration, given
$Q_n(s,a)=\Phi(s,a)^\top\theta_n$, one first computes
\[
V_n(s)=\max_a Q_n(s,a).
\]
Then a mini-batch $\mathcal B_n\subset\mathcal B$ is sampled, and for
each tuple $(s_i,s_i',a_i,r_i)$ in the mini-batch, the
{PhiBE-Q residual} is computed as
\begin{equation}
\label{td-data}
{
\delta_i^{\mathrm{PhiBE}}(\theta_n)
=
r_i
+\mathcal L_{\widehat b_i,\widehat\Sigma_i}V_n(s_i)
-\beta Q_n(s_i,a_i),\quad \widehat b_i=\frac{s_i'-s_i}{\Delta t},
\qquad
\widehat\Sigma_i
=\frac{(s_i'-s_i)(s_i'-s_i)^\top}{\Delta t}.
}
\end{equation}
where $\mathcal L_{b,\Sigma}$ is defined in \eqref{eq:PE}, which is an unbiased model-free estimates of the PhiBE-Q residual at data point $(s_i, a_i)$. {The parameter is updated using the regularized empirical
Gram preconditioner}
\begin{equation}
\label{phibe-q-rule}
{
\theta_{n+1}
=
\theta_n
+
\frac{\alpha_n}{|\mathcal B_n|}
(\widehat G+\eta I)^{-1}
\sum_{(s_i,s_i',a_i,r_i)\in\mathcal B_n}
\Phi(s_i,a_i)\,
\delta_i^{\mathrm{PhiBE}}(\theta_n),
}
\end{equation}
{where $\eta\geq0$ is a small regularization parameter added to improve the conditioning of the empirical Gram matrix.}

This procedure is repeated until a satisfactory approximation of the
Galerkin solution for $Q^*(s,a)$ is obtained. A rigorous analysis of the
sample complexity is left for future work. The algorithm is presented
in Algorithm~\ref{algo2}.

\begin{algorithm}
\caption{{Model-free PhiBE
Q-learning}}
\begin{algorithmic}
\label{algo2}
\STATE \textbf{Input:}
Batch size $|\mathcal B_n|$, initial parameter $\theta_0$, time step
$\Delta t$, discount rate $\beta$, data
$\mathcal B=\{(s_i,s_i',r_i,a_i)\}_{i=1}^N$, basis
$\Phi(s,a)=(\phi_1,\ldots,\phi_n)^\top$, number of iterations
$N_{\mathrm{iter}}$, step sizes $\{\alpha_n\}$,
{and Gram regularization $\eta\geq0$}
\STATE \textbf{Output:}
Estimated parameter $\theta_{N_{\mathrm{iter}}}$ and
$\widehat Q(s,a)=\theta_{N_{\mathrm{iter}}}^\top\Phi(s,a)$

\STATE $\theta\gets\theta_0$

\STATE {Compute the empirical Gram matrix}
\[
{
\widehat G
\gets
\frac{1}{N}\sum_{i=1}^N
\Phi(s_i,a_i)\Phi(s_i,a_i)^\top
}
\]

\FOR{$n=0,\dots,N_{\mathrm{iter}}-1$}
    \STATE Compute the {greedy} state-value function (or the alternative softmax \eqref{eq:data-driven-softmax-value})
    \[
        V(s)\gets\max_a\Phi(s,a)^\top\theta
    \]

    \STATE Sample a mini-batch
    $\mathcal B_n\subset\mathcal B$

    \STATE {For each
    $(s_i,s_i',a_i,r_i)\in\mathcal B_n$, set}
    \[
    {
        \widehat b_i=\frac{s_i'-s_i}{\Delta t},
        \qquad
        \widehat\Sigma_i
        =\frac{(s_i'-s_i)(s_i'-s_i)^\top}{\Delta t}
    }
    \]

    \STATE Compute the PhiBE-Q residual
    \[
    {
    \begin{aligned}
    \delta_i^{\mathrm{PhiBE}}
    ={}&
    r_i
    +\widehat b_i\cdot\nabla_sV(s_i)
    +\frac{1}{2}\widehat\Sigma_i:\nabla_s^2V(s_i)\\
    &-\beta\widehat Q(s_i,a_i)
    \end{aligned}
    }
    \]

    \STATE Form the candidate update
    \[
    {
    \theta^{\mathrm{cand}}
    \gets
    \theta
    +
    \frac{\alpha_n}{|\mathcal B_n|}
    (\widehat G+\eta I)^{-1}
    \sum_{(s_i,s_i',a_i,r_i)\in\mathcal B_n}
    \Phi(s_i,a_i)\delta_i^{\mathrm{PhiBE}}
    }
    \]

    \STATE {Accept
    $\theta\gets\theta^{\mathrm{cand}}$ if the candidate is 
    admissible; otherwise, reduce the step size by backtracking or reject
    the candidate}
\ENDFOR

\STATE \textbf{return}
$\widehat Q(s,a)=
\theta_{N_{\mathrm{iter}}}^\top\Phi(s,a)$
\end{algorithmic}
\end{algorithm}

Several remarks on the algorithm are in order.

First, the algorithm can be viewed as a continuous-time analogue of
classical Q-learning \citep{watkins1992q}. The classical approach
constructs the TD error as the residual of the discrete-time Bellman
equation \eqref{discrete-time be}, whereas our method is based on the
PhiBE-Q equation \eqref{pde of Qhatstar}.
{More precisely, the BE-Q comparator used in our
experiments replaces \eqref{td-data} by}
\begin{equation}
{
\delta_i^{\mathrm{BE}}(\theta_n)
=
\Delta t\,r_i
+
e^{-\beta\Delta t}V_n(s_i')
-
Q_n(s_i,a_i).
}
\label{eq:data-driven-be-q-residual}
\end{equation}
{All other components of the update, including the data,
basis, mini-batches, empirical Gram preconditioner, initialization, and
admissibility test, are kept unchanged, apart from the method-specific
learning rates stated in the numerical experiments.}

Second, the algorithm naturally extends to nonlinear function
approximation, such as deep neural networks. One constructs the same TD
error as in \eqref{td-data} using the parametric form
$Q(s,a;\theta_n)$ and updates the parameters via
\[
\theta_{n+1}
=
\theta_n
+
\alpha_n
\delta(s_i,s_i',a_i,r_i)
\nabla_\theta Q(s_i,a_i;\theta_n).
\]
As in classical RL, convergence guarantees are generally unavailable
under nonlinear function approximation. In practice, stability can be
improved by decoupling the target from the online network. For example,
following \citep{mnih2015human}, one may use a target network and define
$V_n(s)=\max_aQ(s,a;\theta^-)$, where $\theta^-$ denotes the parameters
of a delayed copy of the Q-network.

Third, with the continuous-time formulations of $Q^*(s,a)$ in
\eqref{pde of Qstar} and $Q^\pi(s,a)$ in \eqref{pde of q}, standard
off-policy algorithms including Q-learning \citep{watkins1992q}, DQN
\citep{mnih2015human}, TD3-BC \citep{fujimoto2021minimalist}, and CQL
\citep{kumar2020conservative} can be directly adapted to solve
continuous-time RL problems.

Finally, this algorithm is most directly applicable when
$\max_a\Phi(s,a)^\top\theta$ can be computed efficiently. The numerical
experiments in Section~\ref{sec:lqr-experiments} provide such examples.
For discrete action spaces, the maximum is easy to evaluate. However,
the resulting $V_n(s)$ need not be differentiable. Alternatively, one can approximate $\partial_{s_k}V(s_i)
\approx
\frac{V(s_i+he_k)-V(s_i)}{h}$,
where $e_k$ denotes the $k$th standard basis vector and $h>0$ is a
finite-difference step. 
The convergence of this algorithm with known dynamics is provided in \cite{cho2026policy}.

{A smooth alternative is the entropy-regularized soft
value}
\begin{equation}
{
V_{\theta,\tau}(s)
=
\tau\log\int_{\mathcal A}
\exp\!\left(\frac{Q_\theta(s,a)}{\tau}\right)\,da,
\qquad
\pi_{\theta,\tau}(a\mid s)
=
\frac{\exp(Q_\theta(s,a)/\tau)}
{\int_{\mathcal A}\exp(Q_\theta(s,a')/\tau)\,da'}.
}
\label{eq:data-driven-softmax-value}
\end{equation}
{For a common continuous-time entropy coefficient
$\lambda$, the numerical experiments use}
\begin{equation}
{
\tau_{\mathrm{PhiBE}}
=
\frac{\lambda}{\beta},
\qquad
\tau_{\mathrm{BE}}
=
\lambda\Delta t.
}
\label{eq:data-driven-softmax-temperatures}
\end{equation}
{The second temperature is required because the
sampled-time BE action-value function is scaled by $\Delta t$, whereas
the PhiBE-Q function uses the continuous-time scaling.}

\section{Experiments} \label{sec:4.3}
\subsection{Linear quadratic regulator}
\label{sec:lqr-experiments}
We first consider
the infinite-horizon discounted stochastic control problem
\begin{equation}
\begin{aligned}
    V^*(s)&=\sup_{a_t}\,
    \mathbb E_s\!\left[\int_0^\infty e^{-\beta t}
    \l(-s_t^\top Q_s s_t-a_t^\top R a_t\r)dt\right],\\
    ds_t&=(As_t+Ba_t)\,dt+D\,dB_t, \qquad s_0=s,\\
\end{aligned}
\label{eq:general-lqr-control}
\end{equation}
where $S_t\in\mathbb R^n$, $a_t\in\mathbb R^m$, $Q_s\succeq0$, $R\succ0$, and $\beta>0$. We fixes the LQR notation used throughout the numerical study. The dimensions and
the numerical values of $A,B,D,Q_s,R$, and $\beta$ will be specified
separately in each experiment below.

\paragraph{Exact optimal benchmarks.}
Let $\Sigma=DD^\top$.  If $(A-\frac{\beta}{2}I_n,B)$ is stabilizable and
$(Q_s^{1/2},A-\frac{\beta}{2}I_n)$ is detectable, the stabilizing matrix
$P^*\succeq0$ solves
\begin{equation}
    A^\top P^*+P^*A-P^*BR^{-1}B^\top P^*
    -\beta P^*+Q_s=0.
\label{eq:discounted-care}
\end{equation}
Then the optimal policy $a^*(s) = K^* s$ with $K^*=-R^{-1}B^\top P^*$, the optimal state-action value functions $Q^*(s,a)$ is 
\begin{equation}
\begin{aligned}
    &Q^*(s,a) = c^*
  +s^\top H_{ss}^*s
  +2s^\top H_{sa}^*a
  +a^\top H_{aa}^*a\\
  &c^*=-\frac{\operatorname{tr}(P^*\Sigma)}\beta,\quad 
    H_{ss}^*=-\frac{Q_s+A^\top P^*+P^*A}{\beta},\quad 
    H_{sa}^*=-\frac{P^*B}{\beta},
    \qquad H_{aa}^*=-\frac{R}{\beta}.
\end{aligned}
\label{eq:multidimensional-lqr-solution}
\end{equation}

\paragraph{Quadratic basis and admissible iterates.}
We set $d=n+m$, $z= (s,a)^\top$ and use
the complete quadratic basis $\Phi(s,a)=
    \bigl(1,z_1^2,\ldots,z_d^2,
    2z_1z_2,\ldots,2z_{d-1}z_d\bigr)^\top.$ Thus, $\Phi$ contains one constant, all squared terms, and every distinct
quadratic cross term. Its dimension is $p=1+\frac{d(d+1)}{2}$.
The parametric state-action value function $Q_\th(s,a) = \th^\top \Phi(s,a)$ can be equivalently represented as 
\begin{equation}\label{eq:lqr-Qvalue}
    Q_\th(s,a) = c
  +s^\top H_{ss}s
  +2s^\top H_{sa}a
  +a^\top H_{aa}a, 
\end{equation}
For the unbounded action space $\mathbb R^m$, one requires
$H_{aa}\prec0$ so that the hard maximum $\max_a Q_\th(s,a)$ is finite and unique. The greedy policy is $a_\theta(s)=-H_{aa}^{-1}H_{sa}^\top s$,
and the corresponding maximized value function is and
\begin{equation}
\begin{aligned}
V_\theta(s)&:= \max_a Q_\th(s,a) =c+s^\top M_\theta \quad M_\theta=H_{ss}-H_{sa}H_{aa}^{-1}H_{sa}^\top
\end{aligned}
\label{eq:lqr-greedy-value}
\end{equation}
If a candidate parameter would make this maximum undefined or numerically
unsafe, the entire candidate is rejected.  More precisely, for a fixed
margin $\varepsilon_{\mathrm{conc}}>0$,
\begin{equation}
    \theta_{n+1}=
    \begin{cases}
      \theta_{n+1}^{\mathrm{cand}},
      &\text{if every candidate coefficient is finite and }
      H_{aa,n+1}^{\mathrm{cand}}\preceq
      -\varepsilon_{\mathrm{conc}}I_m,\\
      \theta_n,&\text{otherwise}.
    \end{cases}
\label{eq:lqr-rejection-rule}
\end{equation}
Thus an unsafe iteration leaves the parameter $\th$ unchanged and does not
count toward the convergence patience.

\subsubsection{Known dynamics}
\label{sec:known-dynamics-lqr-experiment}

We first illustrate the continuous-time Q-learning algorithm
\eqref{eq:newalgoite} under known dynamics using the scalar stochastic LQR
problem. The goal of this experiment is to examine the theoretical
guarantees in Theorem~\ref{thm:galerkin_convergence}. We set $m=n=1$ and
\[
A=B=D=Q_s=R=1.
\]

For each iterate, given parameter $\th_n$, the continous-time-Q residual
\[
\delta_n(s,a)
:=
r(s,a)+(\mL V_{\theta_n})(s,a)-\beta Q_{\theta_n}(s,a)
\]
is evaluated exactly, with $V_{\theta_n}$ and $Q_{\theta_n}$ defined as in
\eqref{eq:lqr-greedy-value} and \eqref{eq:lqr-Qvalue}. Given an integration
distribution $\mu$, the parameter update \eqref{eq:newalgoite} is
\begin{equation}
    \theta_{n+1}
    =
    \theta_n+\alpha G^{-1}g_n,
    \qquad
    G=\mathbb E_\mu[\Phi\Phi^\top],
    \qquad
    g_n=\mathbb E_\mu[\delta_n\Phi].
\label{eq:lqr-experiment-update}
\end{equation}
All expectations in $G$ and $g_n$ are computed from exact moments.

\paragraph{Default setting.}
For the default value $\beta=1$, the discounted Riccati equation gives
\[
P^*=\frac{1+\sqrt{5}}{2},
\qquad
K^*=-\frac{1+\sqrt{5}}{2}.
\]
Consequently, the optimal $Q$-function is
\begin{equation}
    Q^*(s,a)
    =
    -\frac{1+\sqrt{5}}{2}
    -(2+\sqrt{5})s^2
    -(1+\sqrt{5})sa-a^2.
\label{eq:scalar-lqr-optimal-q-beta1}
\end{equation}
The default integration distribution is the standard Gaussian
$\mathcal N(0,I_2)$ on $(s,a)$. We use the constant learning rate
$\alpha=0.2$ and the initial function
\begin{equation}
    Q_0(s,a)=-1-s^2-sa-a^2,
\end{equation}
which gives the initial greedy policy $a_0(s)=-0.5s$.
An update is accepted only if all of its entries are finite and its
candidate action coefficient satisfies
$H_{aa,n+1}^{\mathrm{cand}}\leq\varepsilon_{\mathrm{conc}} = -10^{-4}$. The reported error is the
target-normalized Galerkin norm
\begin{equation}
    e_n
    =
    \frac{\|\theta_n-\theta^*\|_G}{\|\theta^*\|_G},
    \qquad
    \|v\|_G=(v^\top Gv)^{1/2}.
\label{eq:lqr-experiment-error}
\end{equation}
The displayed curves are floored at $3\times10^{-16}$ so that zero
and sub-machine-precision errors remain visible on a logarithmic axis.

\begin{figure*}[t]
    \centering
    \includegraphics[width=\textwidth]{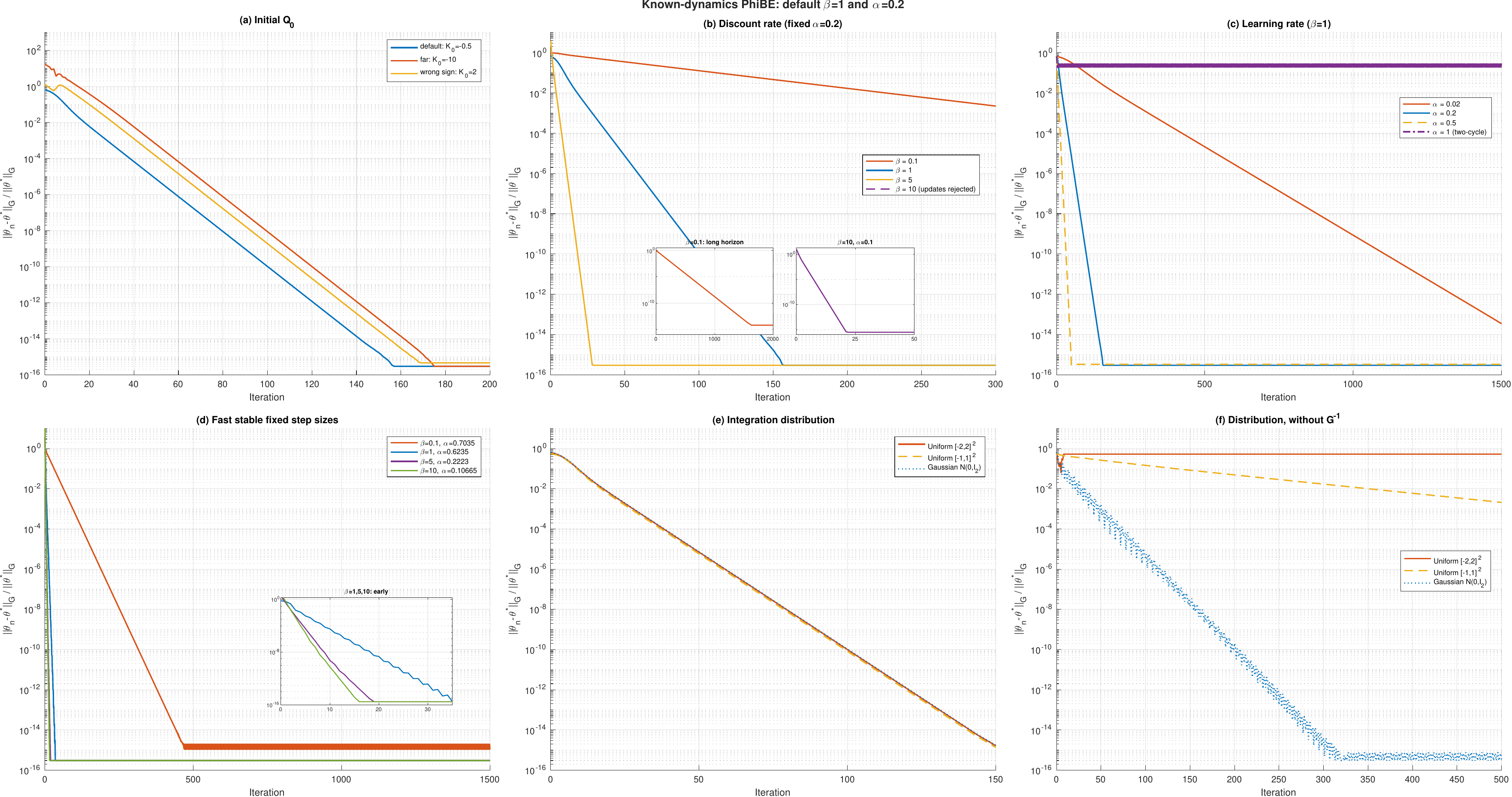}
    \caption{Parameter dependence of the known-dynamics Galerkin iteration
    for the scalar stochastic LQR problem with
    $A=B=D=Q_s=R=1$. Unless a panel varies the corresponding quantity,
    $\beta=1$, $\alpha=0.2$, the integration law is
    $\mathcal N(0,I_2)$, and
    $Q_0=-1-s^2-sa-a^2$. The vertical axis is the normalized $G$-norm
    error in~\eqref{eq:lqr-experiment-error}.
    (a) Dependence on the initial $Q$-function, including a far negative
    gain and a wrong-sign gain.
    (b) Dependence on $\beta\in\{0.1,1,5,10\}$ at fixed $\alpha=0.2$.
    One inset follows the slowly converging $\beta=0.1$ curve over a longer
    horizon, while the other shows that the rejected $\beta=10$ case
    converges after reducing the learning rate to $\alpha=0.1$.
    (c) Dependence on the fixed learning rate
    $\alpha\in\{0.02,0.2,0.5,1\}$ at $\beta=1$; the $\alpha=1$ case forms
    a two-cycle.
    (d) Empirically tuned fast stable pairs
    $(\beta,\alpha)\in
    \{(0.1,0.7035),(1,0.6235),(5,0.2223),(10,0.10665)\}$.
    The inset resolves the early iterations for $\beta\in\{1,5,10\}$.
    (e) Comparison of the uniform laws on $[-2,2]^2$ and $[-1,1]^2$
    with the standard Gaussian law using the exact Gram preconditioner.
    (f) The same distribution comparison without Gram preconditioning.
    The default case is blue in every panel.}
    \label{fig:known-dynamics-parameter-dependence}
\end{figure*}

\paragraph{Dependence on the initial function.}
Panel~(a) compares the default initialization with the farther
initialization $Q_0=-10-101.5s^2-20sa-a^2$ with $K_0=-10$, and $Q_0=-8s^2+8sa-2a^2$ with
$K_0=2$
which has a wrong-sign gain. All three trajectories approach the same
Galerkin solution, which coincides with the exact $Q^*$ because $Q^*$ belongs to the chosen feature space.

\paragraph{Dependence on the discount rate.}
Panel~(b) compares
$\beta\in\{0.1,1,5,10\}$ at the fixed learning rate $\alpha=0.2$.
Among the cases for which the learning rate $\alpha=0.2$ lies in the stable range, the
$\beta=0.1$ trajectory converges substantially more slowly than the
default $\beta=1$ trajectory, whereas the $\beta=5$ trajectory converges
more rapidly. The long-horizon inset confirms that the $\beta=0.1$
trajectory eventually reaches machine precision. For $\beta=10$, every
candidate update violates the imposed action-concavity condition $H_{aa,n+1}^{\mathrm{cand}}
\leq -\varepsilon_{\mathrm{conc}},$
so every update is rejected and the iterate remains unchanged. The
resulting flat curve therefore reflects the safety mechanism rather than
convergence. The second inset shows that reducing the learning rate to
$\alpha=0.1$ restores convergence for $\beta=10$. 
Thus, within the stable
learning-rate range, the experiment indicates that, for a fixed $\alpha$,
increasing $\beta$ accelerates convergence. At the same time, increasing
$\beta$ reduces the range of stable learning rates, so a sufficiently
large $\beta$ requires a smaller value of $\alpha$.

\paragraph{Dependence on the learning rate.}
Panel~(c) varies the learning rate $\alpha$ while keeping all other
parameters fixed at their default values. Small learning rates are stable
but lead to slower convergence. Increasing $\alpha$ from $0.02$ to $0.2$
and then to $0.5$ accelerates convergence. For $\alpha=1$, the parameter
iterates alternate between two distinct values, producing a persistent
two-cycle. This occurs because $\alpha=1$ is too large for the iteration
map to be contractive.

\paragraph{Optimal convergence rate.}
In Panel~(d), for each
$\beta\in\{0.1,1,5,10\}$, we first identify the empirical largest
learning rates by 
coarse-to-fine numerical search that minimizes
the number of iterations required for the normalized $G$-error to remain
below $10^{-8}$ and $10^{-12}$. The experiment shows that the largest stable learning-rate decreases as $\beta$ increases. It also shows that the fastest achievable convergence becomes more rapid as $\beta$ increases. Both observations are consistent with the qualitative predictions of our theorem. 
The contraction bound in the theorem is minimized at $\alpha=1/\beta$. However, this value need not coincide with the empirically optimal learning rate because the theoretical
contraction estimate is generally conservative.

\paragraph{Dependence on the integration distribution.}
Panel~(e) shows essentially overlapping trajectories for the two uniform laws and the Gaussian law. In this LQR example, the Bellman residual $\delta_n$ lies exactly in the quadratic feature space. Hence
$g_n=Gd_n$, where $d_n$ is the coefficient vector of the residual, and
the preconditioner gives $G^{-1}g_n=d_n$. The parameter update is therefore independent of the integration distribution, up to roundoff, as long as $G$ is nonsingular. This overlap relies jointly on exact quadratic representability, exact moment evaluation, and Gram preconditioning.
Panel~(f) repeats the comparison without $G^{-1}$. The update direction
then remains $g_n=Gd_n$, so its scaling and conditioning depend on the
integration law. Consequently, the Gaussian case converges, the uniform
law on $[-1,1]^2$ converges much more slowly, and the uniform law on
$[-2,2]^2$ remains bounded away from the target over the displayed
horizon.

\paragraph{Effect of preconditioning.}
Panels~(e) and~(f) illustrate the role of the Gram preconditioner
$G^{-1}$. Panel~(e) compares the convergence under the uniform
distributions on $[-2,2]^2$ and $[-1,1]^2$ and the standard Gaussian
distribution $\mathcal N(0,I_2)$, with all other parameters fixed at their
default values. When Gram preconditioning is used, the iterations are
stable under all three integration distributions. Panel~(f) repeats the
same experiment without Gram preconditioning. Under the same learning
rate, the Gaussian case converges, the uniform distribution on
$[-1,1]^2$ converges much more slowly, and the uniform distribution on
$[-2,2]^2$ remains bounded away from the target over the displayed
horizon. This comparison demonstrates that $G^{-1}$ normalizes the
Galerkin update and can substantially reduce its sensitivity to the choice
of integration distribution.

The nearly exact overlap in Panel~(e) is a special feature of the present
LQR experiment: the strong-form residual $\delta_n$ belongs exactly to the
quadratic feature space. In particular, if
$\delta_n=\Phi^\top d_n$, then $g_n=Gd_n$, and the preconditioned direction
is therefore $G^{-1}g_n=d_n$. Thus, the factor $G^{-1}$ completely removes
the distribution-dependent scaling and conditioning introduced by $G$ in
this example. With a misspecified basis, an inexact or empirical Gram
matrix, or a residual that does not belong to the feature space,
preconditioning can still improve scaling and conditioning, but it does
not generally make the iteration independent of the integration
distribution.

\subsubsection{Known discrete-time dynamics}
\label{sec:known-discrete-time-dynamics-lqr-experiment}
We next consider the setting in which the continuous-time dynamics are
unavailable and only the discrete-time held-action transition
kernel $\rho_{\dt}(s'|s,a)$ are given.  For the LQR model in
\eqref{eq:general-lqr-control}, this kernel is Gaussian:
\begin{equation}
\begin{aligned}
 \rho_{\dt}(\,\cdot\mid s,a)
 &=
 \mathcal N\!\left(A_{\dt}s+B_{\dt}a,W_{\dt}\right),\quad
 A_{\dt}=e^{A\dt},\quad
 B_{\dt}=\int_0^{\dt}e^{A(\dt-\tau)}B\,d\tau,\\
 W_{\dt}
 &=
 \int_0^{\dt}
 e^{A(\dt-\tau)}\Sigma e^{A^\top(\dt-\tau)}\,d\tau,
 \qquad \Sigma=DD^\top .
\end{aligned}
\label{eq:lqr-known-transition-kernel}
\end{equation}
All figures use the
scalar instance $A=-1$, $B=0.5$, $D=3$, $Q_s=R=1$, $\beta = 0.5$ with $\dt$ specified
separately.

We will use two surrogate models, PhiBE-Q and BE-Q, to approximate the optimal policy. 
\begin{equation}
\begin{aligned}
 \beta Q_{\mathrm{PhiBE},\dt}^*(s,a)
 &=
 r(s,a)
 +
 \mathcal L_{\widehat b_{\dt},0}
 \left[\max_{a'}Q_{\mathrm{PhiBE},\dt}^*(s,a')\right],\\
 Q_{\mathrm{BE},\dt}^*(s,a)
 &=
 r_{\dt}(s,a)
 +
 \gamma\,
 \mathbb E_{s'\sim\rho_{\dt}(\cdot\mid s,a)}
 \left[\max_{a'}Q_{\mathrm{BE},\dt}^*(s',a')\right].
\end{aligned}
\label{eq:lqr-known-transition-surrogate-equations}
\end{equation}
where \[
 \widehat b_{\dt}(s,a)
 = \frac{A_{\dt}-I}{\dt}s+\frac{B_{\dt}}{\dt} a,
 \quad r_{\dt}:=\dt\,r, \quad \gamma:=e^{-\beta\dt}
\]
and $\widehat\Sigma_{\dt}$ is set to zero manually  because the optimal policy of LQR is independent of $\sigma$ (see more details in \citep{zhu2025optimal}).  

The algorithm to find $Q^*_\dt$ that solves \eqref{eq:lqr-known-transition-surrogate-equations} is the same as in Section \ref{sec:known-dynamics-lqr-experiment} except that the parameter of the linear approximation is updated according to the preconditioned form \eqref{eq:lqr-experiment-update} with the residual defined differently,
\begin{equation}
\begin{aligned}
 \delta_n^{\mathrm{PhiBE}}(s,a)
 &=
 r(s,a)
 +\mathcal L_{\widehat b_{\dt},0}V_{\theta_n^{\mathrm{PhiBE}}}(s,a)
 -\beta Q_{\theta_n^{\mathrm{PhiBE}}}(s,a),\\
 \delta_n^{\mathrm{BE}}(s,a)
 &=
 r_{\dt}(s,a)
 +\gamma\,
 \mathbb E\!\left[
 V_{\theta_n^{\mathrm{BE}}}(S')\mid s,a
 \right]
 -Q_{\theta_n^{\mathrm{BE}}}(s,a).
\end{aligned}
\label{eq:lqr-known-transition-updates}
\end{equation}
Both methods use the same standard Gaussian integration law
$\mu=\mathcal N(0,I_2)$, and the same initialization within each
comparison. All Galerkin expectations are evaluated from exact second-
and fourth-order moments. The candidate in
\eqref{eq:lqr-known-transition-updates} is accepted or rejected according
to \eqref{eq:lqr-rejection-rule}. Consequently, each stable iteration
converges to its own finite-$\dt$ surrogate target, which need not equal
the true continuous-time target. 

\begin{figure*}[t]
    \centering
    \includegraphics[width=\textwidth]{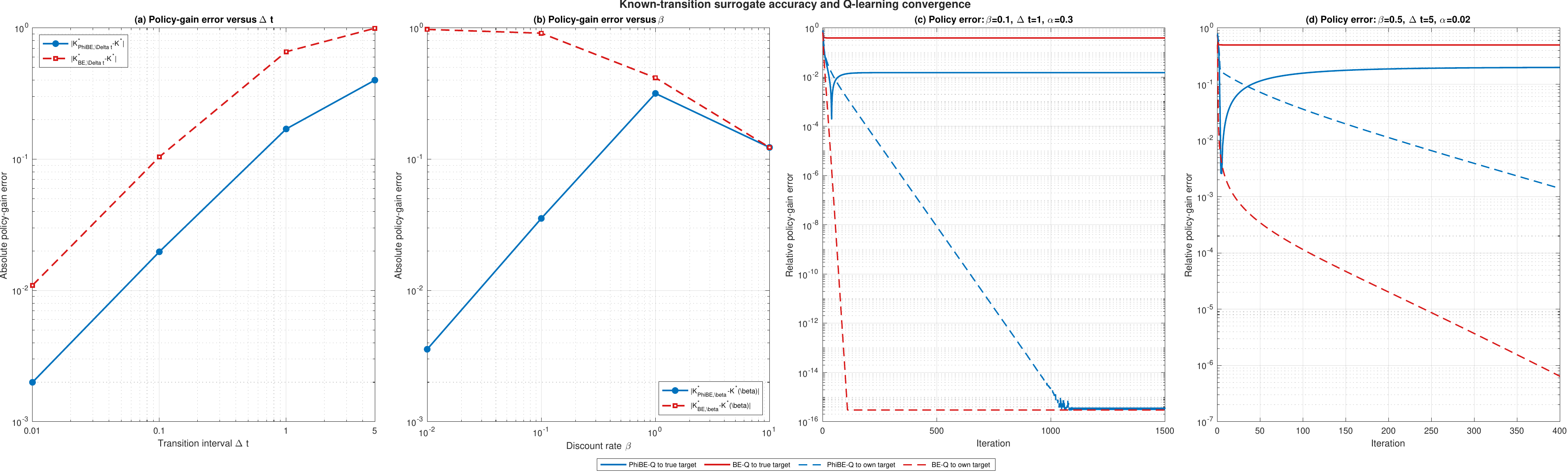}
    \caption{
    Comparison of the PhiBE and BE surrogate models and their associated
    Q-learning algorithms under known discrete-time transition dynamics.
    Panel~(a) reports the absolute errors of the PhiBE and BE optimal
    policy gains relative to the continuous-time optimal gain as the
    observation interval varies over
    $\dt\in\{0.01,0.1,1,5\}$, with $\beta=0.5$.
    Panel~(b) reports the corresponding policy-gain errors as the
    discount rate varies over
    $\beta\in\{0.01,0.1,1,10\}$, with $\dt=1$.
    In Panels~(a) and~(b), the surrogate Riccati equations are solved
    directly, without a Q-learning iteration.
    Panels~(c) and~(d) show the policy-gain errors of the population,
    Gram-preconditioned PhiBE-Q and BE-Q iterations for
    $(\beta,\dt,\alpha)=(0.1,1,0.3)$ and
    $(\beta,\dt,\alpha)=(0.5,5,0.02)$, respectively.
    In these two panels, the solid curves measure error relative to the
    true continuous-time optimal policy, whereas the dashed curves
    measure error relative to the corresponding method's own surrogate
    optimal policy.
    }
    \label{fig:known-transition-comparison}
\end{figure*}

Figure~\ref{fig:known-transition-comparison} separates surrogate-model bias from Q-learning optimization error. Panels~(a) and~(b) compare the optimal feedback gains obtained by directly solving the PhiBE and BE surrogate models with the continuous-time optimum. Both surrogates become more accurate as $\dt$ decreases, while PhiBE's advantage becomes more pronounced for small $\beta$, consistent with the $O(\beta\dt)$ PhiBE and $O(\dt)$ BE approximation errors established in \citet{zhu2025optimal}; the corresponding true-dynamics value functions are reported in Appendix~\ref{appendix:figures}. Panels~(c) and~(d) show that both Q-learning algorithms \eqref{eq:lqr-experiment-update}, \eqref{eq:lqr-known-transition-updates} under the quardratic bases converge to their respective surrogate targets, as indicated by the dashed curves approaching zero. The solid curves measure error relative to the continuous-time optimum and may decrease initially before increasing toward the nonzero surrogate bias. The smaller limiting solid-blue errors show that the PhiBE surrogate policies are closer to the continuous-time optimal policy in the displayed settings.

We next test the population Galerkin solutions when the approximation
space is misspecified in Figure \ref{fig:model-error-galerkin-decomposition}.  We use Gaussian radial functions in the state and
a quadratic action dependence,
\begin{equation}
 \psi_j(s)=\exp\!\left(-\frac{(s-\mu_j)^2}{2\sigma^2}\right),
 \qquad
 Q_\theta(s,a)=\sum_{j=1}^{N}\psi_j(s)
 \bigl(c_j+2g_j a+h_j a^2\bigr).
 \label{eq:model-error-gaussian-quadratic-basis}
\end{equation}
For Figure~\ref{fig:model-error-galerkin-decomposition}, we set $N=6$,
$\sigma=0.8$, and place the centers uniformly on $[-5,5]$, so that
$\mu_j=-5+10(j-1)/(N-1)$. We sweep
$\dt\in\{0.01,0.02,0.05,0.1,0.2,0.5,1\}$. The two dash lines in Figure \ref{fig:model-error-galerkin-decomposition} shows that the best-approximation errors of both  surrogate Q-functions under the new bases remain approximately constant over the sweep, allowing us to test
the Galerkin solutions of the two surrogate models under a fixed level of
basis-induced model error.

\begin{figure*}[t]
 \centering
 \includegraphics[width=\textwidth]{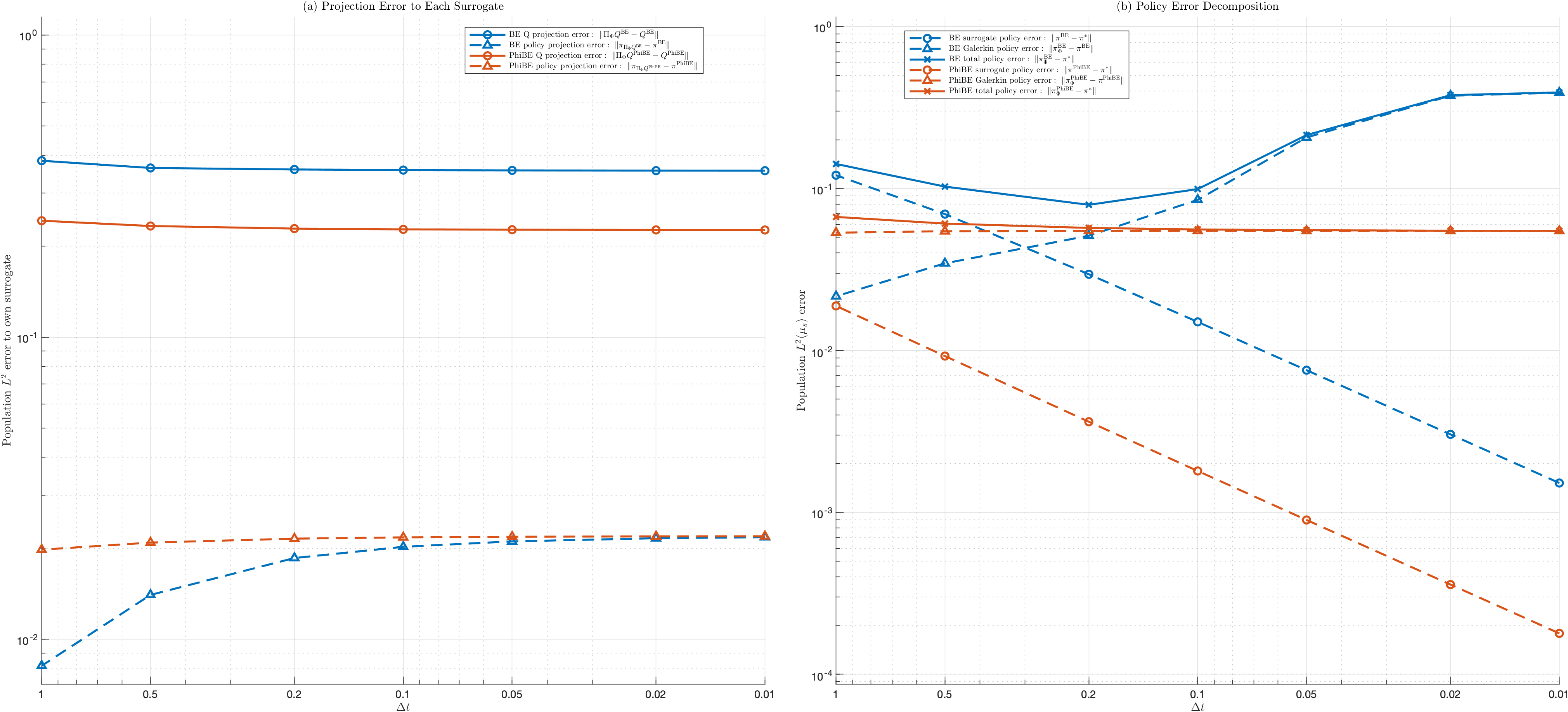}
 \caption{Galerkin approximation under model error for the BE and PhiBE
 surrogate equations.  Blue curves denote BE and orange curves denote
 PhiBE; within each panel, line style and marker encode the error type.
 Panel~(a) reports the normalized population projection errors of each
 exact finite-$\dt$ surrogate Q-function and of the greedy policy induced
 by that projection, both measured relative to the same surrogate.
 Panel~(b) decomposes the population $L^2(\mu)$ Q-function error into
 surrogate error $\|Q^{j}-Q^*\|$, Galerkin error
 $\|Q_\Phi^{j}-Q^{j}\|$, and total error
 $\|Q_\Phi^{j}-Q^*\|$, for
 $j\in\{\mathrm{BE},\mathrm{PhiBE}\}$.  Panel~(c) gives the analogous
 decomposition for the greedy policies in $L^2(\mu_s)$, where
 $\mu_s=\mathcal N(0,1)$.  Here $A=-1$, $B=0.5$, $D=3$,
 $Q_s=R=1$, $\beta=0.5$, $N=6$, $\sigma=0.8$, and the Gaussian centers
 are uniformly spaced on $[-5,5]$.}
 \label{fig:model-error-galerkin-decomposition}
\end{figure*}

Figure~\ref{fig:model-error-galerkin-decomposition} shows that, as claimed in Section~\ref{sec:off-policy-rl}, BE-Q is more sensitive to model error as $\dt\to0$. In Panel~(a), the two solid curves report the best-approximation errors of the BE-Q and PhiBE-Q surrogate solutions under the basis functions defined above.  Although the Q-function projection errors remain approximately constant across $\dt$, the policy error induced by the projected BE-Q solution increases as $\dt\to0$, whereas its PhiBE-Q counterpart remains nearly constant. This behavior arises because the action-dependent signal in BE-Q that determines the greedy policy is only $O(\dt)$, which can be increasingly masked by the constant model error.
Panel~(b) shows that the Galerkin equation further amplifies this effect.  As shown previously in Figure~\ref{fig:known-dynamics-parameter-dependence}, the surrogate errors decrease as $\dt$ decreases. However, the BE-Q Galerkin error increases as $\dt\to0$, while the corresponding PhiBE-Q error remains approximately constant. The resulting total BE-Q policy error therefore increases despite the decreasing surrogate-model error as $\dt \to 0$. This instability arises because the BE-Q residual contains the term $\max_{a'}Q_\Phi^{\mathrm{BE}}(s',a')$. When the action-dependent signal is only $O(\dt)$, as in BE-Q, a small approximation error can cause large changes in the maximizing action and, consequently, in both the Galerkin solution and its greedy policy. By contrast, PhiBE uses a scaled-Q equation in which the action-dependent signal remains order one. Its Galerkin policy error is therefore much less sensitive to decreasing $\dt$ in this experiment.

\begin{figure*}[t]
    \centering
    \begin{minipage}[t]{0.495\textwidth}
        \centering
        \includegraphics[width=\linewidth]
        {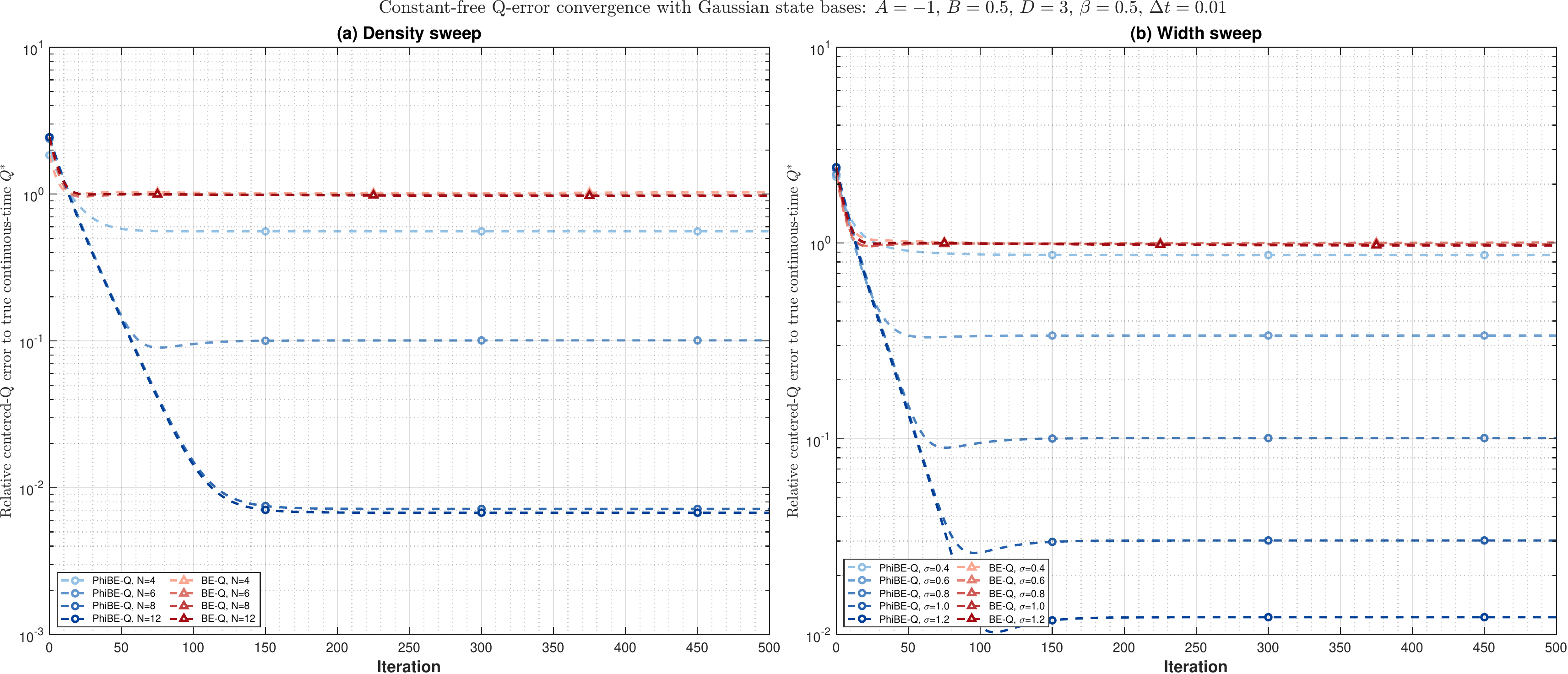}
    \end{minipage}\hfill
    \begin{minipage}[t]{0.495\textwidth}
        \centering
        \includegraphics[width=\linewidth]
        {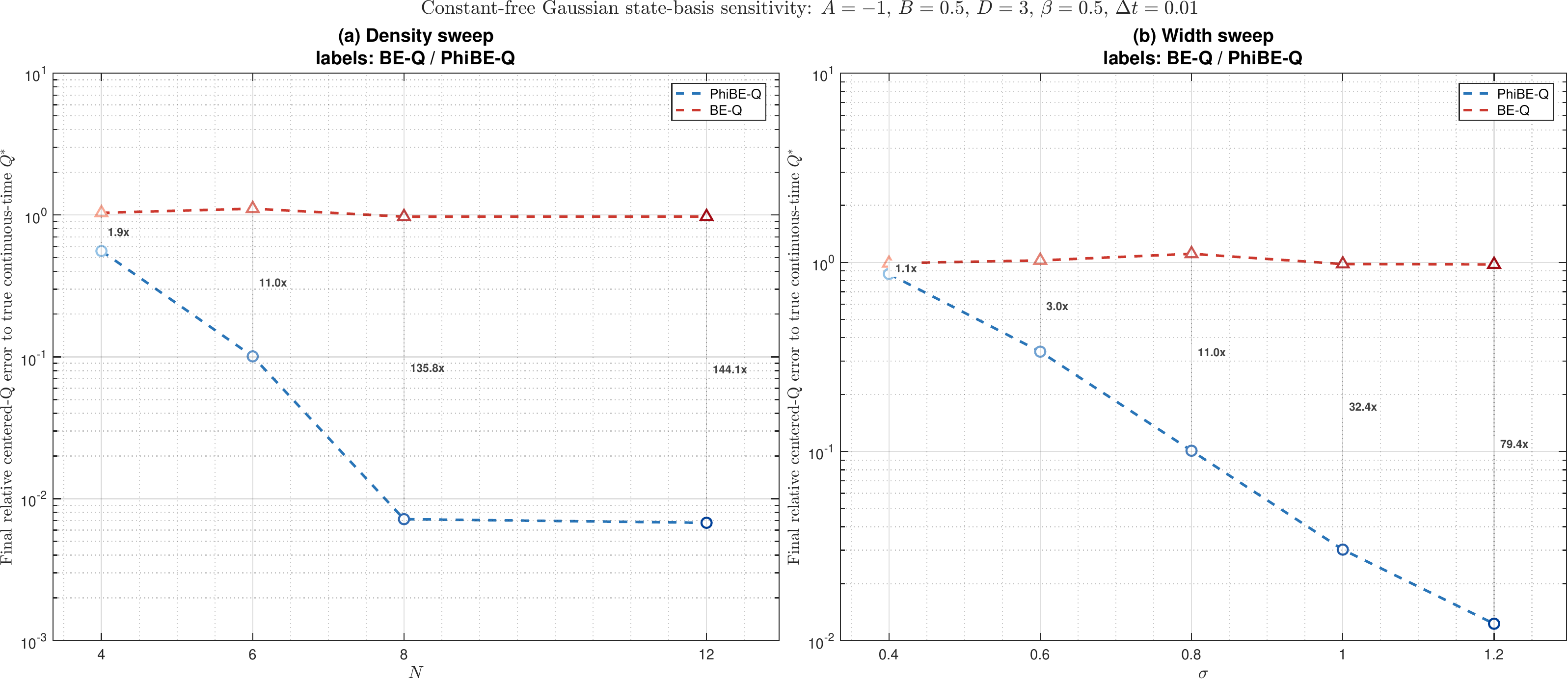}
    \end{minipage}
    \caption{
    Sensitivity of PhiBE-Q and standard Q-learning based on the Bellman
    equation (BE-Q) to the Gaussian state basis. The parameters are
    $A=-1$, $B=0.5$, $D=3$, $Q_s=R=1$, $\beta=0.5$, and
    $\dt=0.01$. All curves report the centered relative error
     with respect to the true continuous-time
    optimal Q-function $Q^*$.
    In the density sweep, the width is fixed at $\sigma=0.8$ and
    $N\in\{4,6,8,12\}$; in the width sweep, $N=6$ is fixed and
    $\sigma\in\{0.4,0.6,0.8,1.0,1.2\}$.
    The left block shows the first $500$ iterations, while the right
    block reports the errors after $1500$ iterations.
    PhiBE-Q is shown by blue circles and BE-Q by red triangles, with
    darker shades corresponding to larger $N$ or $\sigma$.
    The labels in the right block give the ratio
    $e_Q^{\mathrm{BE}}/e_Q^{\mathrm{PhiBE}}$; hence, a ratio greater
    than one indicates a smaller error for PhiBE-Q.
    }
    \label{fig:gaussian-basis-sensitivity}
\end{figure*}

Figure~\ref{fig:gaussian-basis-sensitivity} compares the performance of
PhiBE-Q and BE-Q algorithms \eqref{eq:lqr-known-transition-surrogate-equations} under Gaussian state bases of different densities and
widths. PhiBE-Q achieves a smaller Q-function error in every
displayed case. PhiBE-Q has a smaller Q-function
error in every case: its terminal error decreases from
$5.58\times10^{-1}$ at $N=4$ to $7.17\times10^{-3}$ at $N=8$, and from
$8.68\times10^{-1}$ at $\sigma=0.4$ to $1.23\times10^{-2}$ at
$\sigma=1.2$, while the BE-Q error remains near one. Thus, richer and smoother Gaussian representations
generally reduce the approximation error of PhiBE-Q in this example. While BE-Q error is always large due to its instability when $\dt$ is small. The limited
improvement from $N=8$ to $N=12$ is due to the ill-conditioning of $G$
caused by closely spaced Gaussian centers, which can slow or destabilize
the iteration, as indicated by our convergence theorem.


\subsubsection{Data-driven algorithm}
\label{sec:data-driven-lqr-experiment}

We next compare data-driven PhiBE-Q learning with standard Q-learning
based on the discrete-time Bellman equation, denoted by BE-Q. In both
experiments, the data consist of independent off-policy tuples
$\{(s_i,a_i,s_i',r_i)\}_{i=1}^{N_{\rm data}}$, where the action is held
constant over $[0,\dt]$, $s_i'$ is sampled from the exact held-action
transition kernel, and $r_i=r(s_i,a_i)$. Within each paired replication,
PhiBE-Q and BE-Q use the same data and mini-batches.

\paragraph{Quadratic basis and hard maximization.}
We first use the scalar LQR dynamics and complete quadratic basis from
Section~\ref{sec:known-dynamics-lqr-experiment},
\[
 A=B=D=Q_s=R=1,\qquad
 \beta=0.5,\qquad
 \dt=0.1, \quad Q_\theta(s,a)
 =c+H_{ss}s^2+2H_{sa}sa+H_{aa}a^2.
\]
The off-policy samples satisfy
\[
 (s_i,a_i)
 \sim\mathcal N\!\left(
 0,\operatorname{diag}(1.5^2,3^2)\right),
 \qquad
 s_i'\mid(s_i,a_i)
 \sim\rho_{\dt}(\cdot\mid s_i,a_i).
\]
We use $N_{\rm data}=10^5$. We implement Algorithm \ref{algo2} with regularizer $\eta
 =
 10^{-8}\frac{\operatorname{tr}(\widehat G)}
 {\dim(\theta)}$, mini-batches of size $2000$, $3000$
iterations, and $30$ paired replications. Both methods start from $Q_0(s,a)=-1-s^2-sa-a^2$.

For PhiBE-Q, we manually set $\widehat\Sigma=0$ and use
\begin{equation}
 \delta_i^{\mathrm{PhiBE}}(\theta)
 =
 r_i
 +\frac{s_i'-s_i}{\dt}\,
  \partial_sV_\theta(s_i)
 -\beta Q_\theta(s_i,a_i),
 \qquad
 V_\theta(s)=\max_a Q_\theta(s,a).
 \label{eq:data-driven-phibe-residual}
\end{equation}
The BE-Q comparator uses the same sampler, feature map, empirical Gram
preconditioner, initialization, mini-batches, and admissibility test,
but replaces the residual by
\begin{equation}
 \delta_i^{\mathrm{BE}}(\theta)
 =
 \dt r_i
 +e^{-\beta\dt}V_\theta(s_i')
 -Q_\theta(s_i,a_i).
 \label{eq:data-driven-be-residual}
\end{equation}
The learning-rate schedule is
\[
 \alpha_n
 =
 \frac{\alpha_0}{\sqrt{1+(n-1)/2000}},
\]
with $\alpha_0=0.02$ for PhiBE-Q and $\alpha_0=0.2$ for BE-Q.
A candidate with a nonfinite coefficient or
$H_{aa}>-10^{-4}$ is rejected.

\paragraph{Gaussian state basis and soft policies.}
We next use the model from
Section~\ref{sec:known-discrete-time-dynamics-lqr-experiment},
\[
 A=-1,\qquad B=0.5,\qquad D=3,\qquad Q_s=R=1,
 \qquad \beta=0.5,\qquad \dt=0.01.
\]
With finitely many samples, the state-dependent action curvature in
\eqref{eq:model-error-gaussian-quadratic-basis} can be perturbed close
to zero or become positive. The hard maximum is then highly sensitive
or undefined, while derivatives of $V_\theta$ contain terms proportional
to $1/h_\theta$ and $1/h_\theta^2$. Rejection and backtracking can
therefore stall the hard-max iteration. We instead use the softmax alternative in
Remark~\ref{rmk:softmax}. Softmax removes the nonsmooth switching of
maximizers, but it still requires a finite action partition function.
Under the original parameterization,
$h_\theta(s)=\sum_jh_j\psi_j(s)$ tends to zero as $|s|\to\infty$.
Consequently, the soft-policy variance may diverge in the tails.
We therefore use Gaussian bases only for the state-dependent constant
and linear-action coefficients and use one shared negative quadratic
action coefficient:
\begin{equation}
 \begin{aligned}
 Q_\theta(s,a)
 &=\underbrace{\sum_{j=1}^{N}c_j\psi_j(s)}_{f_\th(s)}+2\underbrace{\sum_{j=1}^{6}g_j\psi_j(s)}_{g_\th(s)}a+ha^2,
 \qquad h<0.
 \end{aligned}
 \label{eq:data-driven-soft-gaussian-basis}
\end{equation}
For a temperature $\tau>0$, the soft value and associated policy are
available analytically:
\begin{equation}
 \begin{aligned}
 V_{\theta,\tau}(s)=
 f_\theta(s)-\frac{g_\theta(s)^2}{h}
 +\frac{\tau}{2}
 \log\!\left(\frac{\pi\tau}{-h}\right),\quad
 \pi_{\theta,\tau}(\cdot\mid s)
 &=
 \mathcal N\!\left(
 -\frac{g_\theta(s)}{h},
 -\frac{\tau}{2h}\right).
 \end{aligned}
 \label{eq:data-driven-soft-policy}
\end{equation}

We use the same entropy coefficient $\lambda$ in the
regularized-HJB objective. Because the continuous-time PhiBE
$Q$-function and the sampled-time BE $Q$-function have different
scalings, their corresponding temperatures are
\begin{equation}
 \tau_{\mathrm{PhiBE}}
 =\frac{\lambda}{\beta}=0.2,
 \qquad
 \tau_{\mathrm{BE}}
 =\lambda\dt=10^{-3}.
 \label{eq:data-driven-two-temperatures}
\end{equation}
The regularized residuals are therefore
\begin{equation}
 \begin{aligned}
 \delta_i^{\mathrm{PhiBE},\lambda}(\theta)
 &=
 r_i+\frac{s_i'-s_i}{\dt}
 \partial_sV_{\theta,\tau_{\mathrm{PhiBE}}}(s_i)
 -\beta Q_\theta(s_i,a_i),\\
 \delta_i^{\mathrm{BE},\lambda}(\theta)
 &=
 \dt r_i+e^{-\beta\dt}
 V_{\theta,\tau_{\mathrm{BE}}}(s_i')
 -Q_\theta(s_i,a_i).
 \end{aligned}
 \label{eq:data-driven-soft-residuals}
\end{equation}

For the entropy-regularized LQR problem, the optimal policy is
\begin{equation}
 \pi_\lambda^*(\cdot\mid s)
 =
 \mathcal N\!\left(
 K^*s,\frac{\lambda}{2R}\right),
 \qquad
 K^*=-R^{-1}B^\top P^*.
 \label{eq:regularized-lqr-optimal-policy}
\end{equation}
We report the expected forward KL divergence and the centered relative Q-error.

The formal experiment sets $\lambda = 0.1$ and uses $N_{\rm data}=10^8$ streamed samples per
replication, mini-batches of size $10^5$, $1000$ iterations.  The initial learning rates are
$\alpha_{0,\mathrm{PhiBE}}=0.01$ and
$\alpha_{0,\mathrm{BE}}=0.001$. All other hyperparameters are the same as before. Although $\theta_0$ is identical, the initial KL errors differ because the two Q-scalings imply different policy temperatures.

\begin{figure*}[t]
    \centering
    \begin{minipage}[t]{0.495\textwidth}
        \centering
        \includegraphics[width=\linewidth]
        {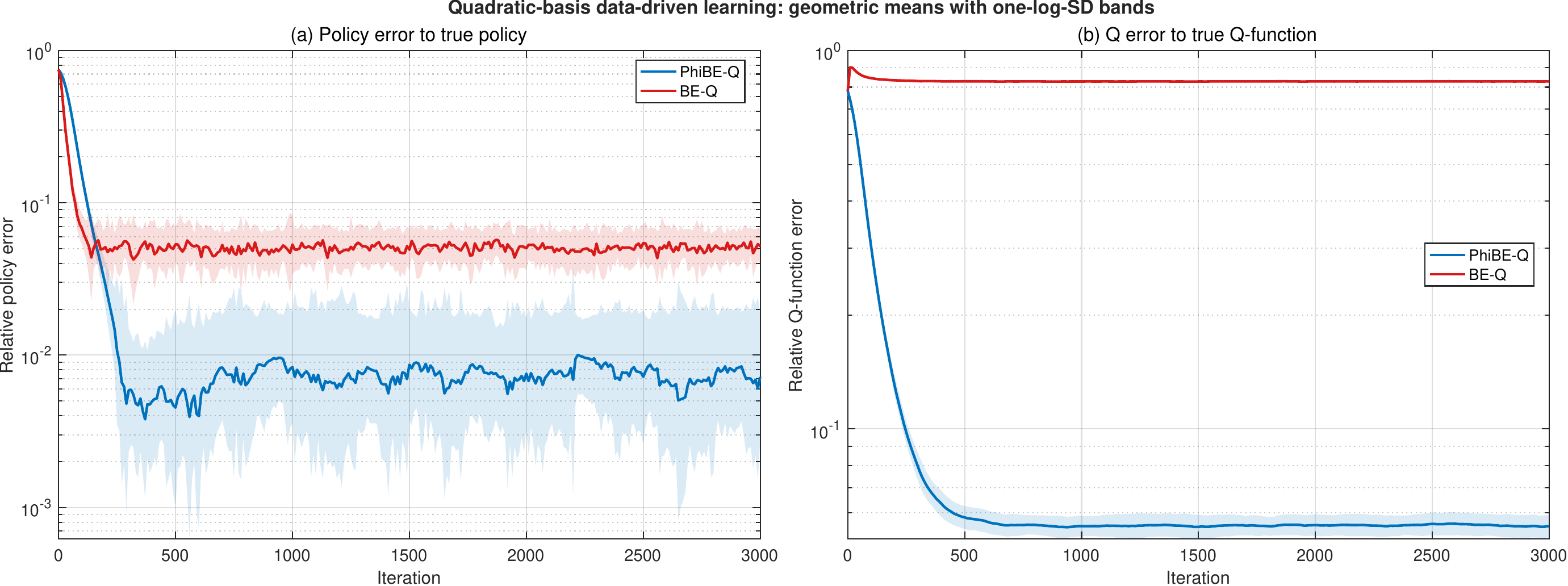}
    \end{minipage}\hfill
    \begin{minipage}[t]{0.495\textwidth}
        \centering
        \includegraphics[width=\linewidth]
        {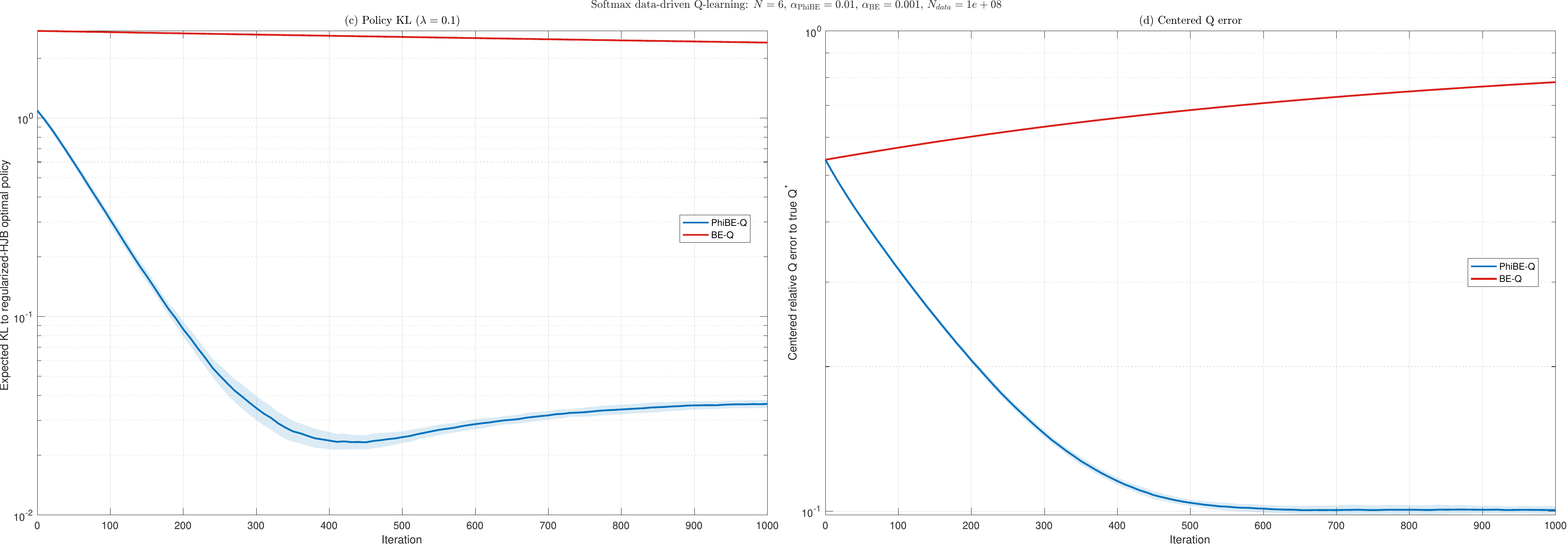}
    \end{minipage}
    \caption{
    Data-driven PhiBE-Q and BE-Q learning.
    Panel~(a) reports the relative policy error to the true
    continuous-time optimal policy under the complete quadratic basis.
    Panel~(b) reports the relative full Q-function error to the true
    continuous-time Q-function under the same basis; this legacy metric
    includes the additive constant.
    Panel~(c) reports the expected forward KL divergence to the
    entropy-regularized HJB optimal Gaussian policy under the Gaussian
    state basis~\eqref{eq:data-driven-soft-gaussian-basis}.
    Panel~(d) reports the centered relative Q-error
     in the same soft-policy
    experiment. Curves are geometric means over $30$ paired replications,
    and shaded regions show one logarithmic standard deviation.
    }
    \label{fig:data-driven-lqr-comparison}
\end{figure*}

Panels~(a) and~(b) show that, under the complete quadratic basis,
data-driven PhiBE-Q is more accurate relative to the continuous-time
targets than BE-Q. 
Panels~(c) and~(d) show that the advantage persists with the
misspecified Gaussian state representation and entropy-regularized
policies. Thus, in this experiment, PhiBE-Q more accurately
recovers both the complete regularized-HJB Gaussian policy and the
policy-relevant shape of the continuous-time Q-function. 

\subsection{CartPole}
\label{sec:deterministic-cartpole}
We next compare data-driven PhiBE-Q-learning with discrete-time
Q-learning on a nonlinear deterministic CartPole problem. The purpose
is to examine whether a policy learned from coarsely sampled
off-policy transitions can benefit from more frequent feedback
during deployment.

\paragraph{Problem setting.}
The state is $s=(\theta,v)$, where $\theta$ is the pole angle measured
from the upright position and $v=\dot\theta$ is its angular velocity.
We consider the reduced angular dynamics
\begin{equation}
\begin{aligned}
\dot\theta &= v,\\
\dot v &= f(\theta,v,a)
:= \frac{
9.8\sin\theta
-0.05v^2\sin(2\theta)
-0.1a\cos\theta
}{
\frac{2}{3}-0.1\cos^2\theta
},
\qquad a\in\mathcal A:=\{-50,0,50\}.
\end{aligned}
\label{eq:cartpole-dynamics}
\end{equation}
Here $a$ denotes the applied force, and no process noise is added.
An episode terminates when the pole reaches the horizontal
configuration,
\[
T_{\mathrm{hor}}
:=\inf\{t\geq0:|\theta_t|\geq\pi/2\}.
\]
The control objective is
\begin{equation}
V^*(s)
=
\sup_{\pi}
\mathbb E_s^\pi
\left[
\int_0^{T_{\mathrm{hor}}}
e^{-\beta t}r(\theta_t)\,dt
\right],
\qquad \beta=0.1,
\label{eq:cartpole-objective}
\end{equation}
with zero continuation value after termination. We use the smooth
reward
\begin{equation}
r(s)
=
0.01+0.99
\exp\left[
-3\left(\frac{\theta}{\theta_{\mathrm{ref}}}\right)^2
\right],
\qquad
\theta_{\mathrm{ref}}=\frac{\pi}{15}.
\label{eq:cartpole-reward}
\end{equation}
Thus, the reward is maximal at the upright position and decreases
smoothly toward $0.01$ as the absolute angle increases. The reference
angle is $12^\circ$. It determines the reward scale, not the termination boundary.

\paragraph{Fixed off-policy data.}
Both methods use the same batch of $M=2000$ independent transitions.
For each transition, the initial state is sampled uniformly from
$[-0.05,0.05]^2$. With probability $0.35$, the behavior action is
sampled uniformly from $\mathcal A$; otherwise, it is chosen by
\[
a_{\mathrm b}(\theta,v)
=
\begin{cases}
50, & \theta+0.25v\geq0,\\
-50, & \theta+0.25v<0.
\end{cases}
\]
This behavior rule is fixed and independent of the learned policies.

The dynamics are simulated using semi-implicit Euler with internal
step $h=0.002$ seconds:
\[
v_{k+1}=v_k+h f(\theta_k,v_k,a),
\qquad
\theta_{k+1}=\theta_k+h v_{k+1}.
\]
Each sampled action is held fixed for at most $N=100$ internal
steps. Hence, the nominal collection interval is
$\Delta t=Nh=0.2$ seconds. Collection stops earlier only if the pole
reaches the horizontal boundary. A transition records
\[
(s_i,a_i,s_i',r_i),
\quad r_i=r(s_i).
\]

\paragraph{Basis and learning protocol.}
We use the constant and Gaussian radial basis functions, $\phi(s)=(1,\{\psi_c(s)\}_{c\in\mathcal C})^\top\in\mathbb R^{10}$
\[
\psi_c(\theta,v)
=
\exp\left[
-\frac12
\left(
\frac{(\theta-c_\theta)^2}{(\pi/4)^2}
+(v-c_v)^2
\right)
\right], \quad c = (c_\theta,c_v)\in\mathcal C = \{-\pi/4,0,\pi/4\}\times\{-1,0,1\}.
\]
we approximate the action-value function by separating coefficient blocks for the three actions 
\[
\hat{Q}_\th(s,a)=\th_a^\top\phi(s), \quad \forall a\in \{0,-50, 50\},
\]
which gives $30$ parameters in total.

We apply the data-driven update
\eqref{phibe-q-rule} from Section~\ref{sec:4.2.2}, with the following
implementation choices. First, we omit Gram preconditioning because the empirical
Gram matrix is severely ill-conditioned over the localized
training domain. Second, the maximum is replaced during
training by the differentiable soft value
\[
V_{\th,\tau}(s)
=
\tau\log\sum_{a\in\mathcal A}
\exp\left(\frac{Q_\th(s,a)}{\tau}\right),
\qquad
\tau_{\mathrm{PhiBE}}=\frac{\lambda}{\beta},
\quad
\tau_{\mathrm{BE}}=\lambda\Delta t.
\]
Third, we set
$\widehat\Sigma=0$, consistent with the deterministic dynamics.

The discrete-time Q-learning baseline uses the same features,
data, and soft value, just by replacing $\delta_i$ in \eqref{phibe-q-rule} with the BE-Q residual
\begin{equation}
\delta^{RL}_i(s_i, s_i',a_i, r_i)
 = \Delta t r_i
+e^{-\beta\Delta t}V_{\th,\tau_{\text{BE}}}(s_i')
-Q_{\th}(s_i,a_i).
\label{eq:cartpole-q-update}
\end{equation}
Both methods share the dataset, basis functions, minibatch
sequence, and zero initialization, same batch size $\mathcal B_k=32$, same $\lambda=0.005$ and use $10^4$ updates
with $\alpha=0.002$

\paragraph{Evaluation protocol.}
After training, each coefficient vector is fixed, and the greedy
policy
\[
\widehat\pi(s)\in\operatorname*{arg\,max}_{a\in\mathcal A}
Q_\vartheta(s,a)
\]
is evaluated without further learning. We vary the deployment
interval over
\[
\delta_{\mathrm{eval}}
\in\{100h,10h,h\}
=
\{0.2,0.02,0.002\}\ \mathrm{s}.
\]
At each decision time, the action is recomputed and then held
constant until the next decision. The integration step remains
$h$ in every case.

For each deployment interval, we simulate $100$ episodes with
initial states drawn uniformly from $[-0.05,0.05]^2$ and paired
across the two methods. Evaluation stops at the horizontal
boundary or at $H=10$ seconds. We report the capped survival time
$\widetilde T=\min(T_{\mathrm{hor}},H)$ and the trajectory-average
absolute angle. Angles are sampled every internal step, irrespective of the
deployment interval.

\begin{figure}[tbp]
    \centering
    \begin{subfigure}{\textwidth}
        \centering
        \includegraphics[width=\linewidth]
        {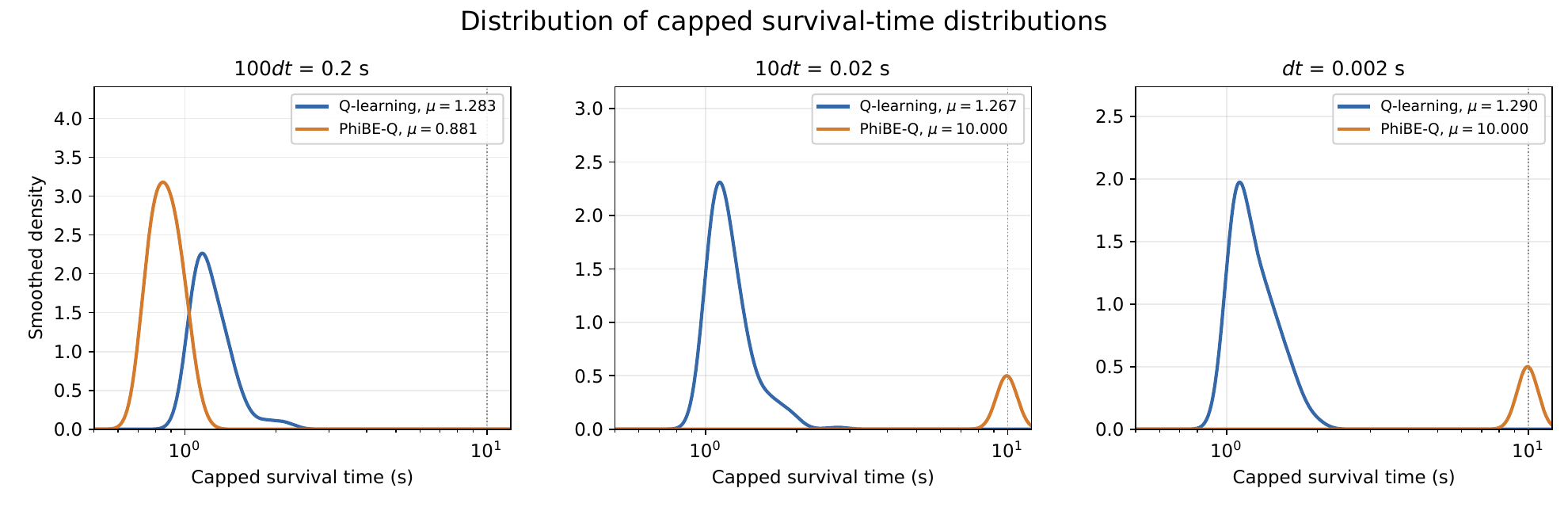}
    \end{subfigure}
    \begin{subfigure}{\textwidth}
        \centering
        \includegraphics[width=\linewidth]
        {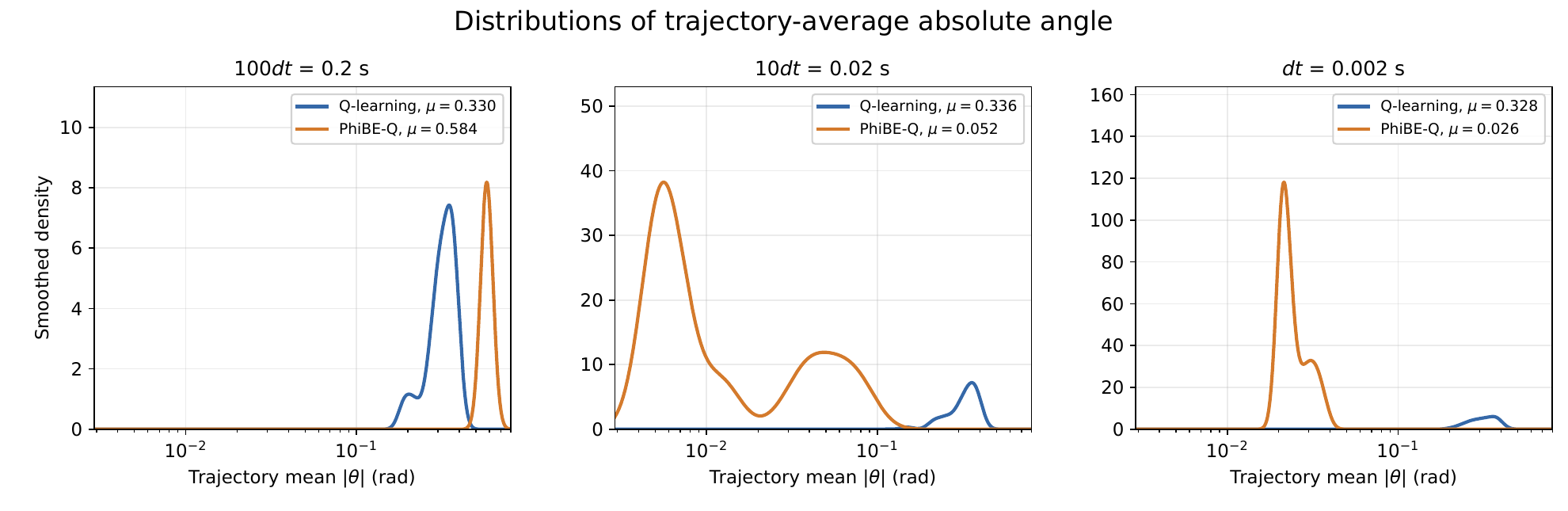}
    \end{subfigure}
    \caption{CartPole: Both methods use the same $2000$ off-policy
    transitions collected every $100dt=0.2$  and the same
    $30$-parameter basis. In each row, the deployment intervals
    are $100dt$, $10dt$, and $dt$ from left to right, where
    $dt=0.002$ s.
    Blue and orange curves represent Q-learning and PhiBE-Q,
    respectively. 
    }
    \label{fig:cartpole-deterministic-comparison}
\end{figure}

\begin{table}[tbp]
\centering
\small
\caption{
Deterministic CartPole performance over $100$ evaluation episodes. Action-update intervals are expressed as multiples of
$dt=0.002$ s.
The final column gives the percentage reaching the $10$-second
evaluation horizon.
}
\label{tab:cartpole-deterministic}
\begin{tabular}{llccc}
\toprule
Action frequency & Method
& Survival (s) & Mean absolute angle (rad) & Reach $10$ s \\
\midrule
$100dt$ & Q-learning
& $1.283\pm0.245$ & $0.3297\pm0.0493$ & $0\%$ \\
& PhiBE-Q
& $0.881\pm0.095$ & $0.5836\pm0.0146$ & $0\%$ \\
\midrule
$10dt$ & Q-learning
& $1.267\pm0.275$ & $0.3359\pm0.0548$ & $0\%$ \\
& PhiBE-Q
& $10.000\pm0.000$ & $0.0518\pm0.0293$ & $100\%$ \\
\midrule
$dt$ & Q-learning
& $1.290\pm0.233$ & $0.3280\pm0.0528$ & $0\%$ \\
& PhiBE-Q
& $10.000\pm0.000$ & $0.0257\pm0.0054$ & $100\%$ \\
\bottomrule
\end{tabular}
\end{table}

\paragraph{Results and interpretation.}
As shown in Figure~\ref{fig:cartpole-deterministic-comparison} and
Table~\ref{tab:cartpole-deterministic}, when actions are updated at the same frequency as the data collection frequency, neither learned
policy maintains balance. In contrast, deploying the same PhiBE-Q policy more frequent, every $0.02$ s
or $0.002$ s keeps $|\theta_t|<\pi/2$ throughout the full
$10$-second evaluation horizon for every evaluated trajectory.
Its mean trajectory-average absolute angle decreases to
$0.0518$ rad and $0.0257$ rad, respectively.
While the Q-learning policy remains unsuccessful under more frequent
deployment, with mean survival times near $1.3$ s.

These results illustrate two phenomena consistent with the motivation of the necessarity of continuous-time Q. First, more frequent feedback can improve
control performance in continuously evolving systems, as
demonstrated by the PhiBE-Q policy. Second, the discrete-time
Q-learning policy shows little improvement under more frequent
deployment, whereas the PhiBE-Q policy benefits substantially.
This contrast is consistent with the
action-dependent signal: it diminishes with the sampling interval
in the discrete-time $Q$-function but remains order one in the
proposed continuous-time formulation. Thus, even when trained on
coarse off-policy data, PhiBE-Q can produce a policy that benefits
from finer feedback during deployment, without collecting
additional training data or retraining.

\section{Conclusion}
In this work, we propose a new continuous-time Q function, whose governing function is a closed-form PDE with single unknown function. It naturally enables off-policy data usage in continuous-time RL settings. By introducing an time-evolutionary equation and leveraging the Galerkin projection, through PhiBE, we develop a stable iterative scheme for solving continuous-time optimal control problems in a model-free way. 

We further establish error bounds and convergence rates of the proposed algorithms under linear function approximation and known dynamics. Finite-data and modified implementations are assessed numerically; the reported comparisons illustrate advantages over discrete-time Q-learning in the tested settings.

Several promising directions remain open for future investigation. First, from a theoretical perspective, some of the assumptions used to establish the current guarantees may potentially be relaxed. While this paper focuses on error bound and convergence rate analysis, establishing theoretical sample complexity bounds is an important next step. Second, extending our methodology to large-scale problems via deep neural approximators presents both opportunities and challenges in expressiveness and stability. Third, incorporating exploration strategies and safety constraints would broaden applicability in real-world control systems. 

Overall, this work provides a principled and data-driven foundation for continuous-time reinforcement learning for off-policy data, and we believe the proposed framework will facilitate further developments toward scalable and reliable continuous-time decision-making.

\appendix
\section{Proof of Theorem \ref{thm:q-pi-property}}\label{sec:proof of thm-q-pi-property}

\begin{proof}
\begin{enumerate}
    \item 
Under Assumption~\ref{ass} and the admissibility of $\pi$, the controlled
diffusion is well posed, and its value function $V^\pi$ defined in \eqref{def of V-pi} is the unique bounded
classical solution of the policy-evaluation equation  \eqref{eq:pe-V} \citep{FS06,YZ99,P09}. 

Averaging the continuous-time Bellman equation \eqref{def of Q-pi}  with respect to $\pi(\cdot\mid s)$ 
\[
\mathbb{E}_a[Q^\pi(s,a)]
=
\frac{1}{\beta}
\left(
r^\pi(s)+(\mathcal L_{b^\pi,\Sigma^\pi} \mathbb{E}_a[Q^\pi(s,a)] )(s)
\right),
\]
one has that $\mathbb{E}_a[Q^\pi(s,a)]$ and $V^\pi$ satisfy the same
policy-evaluation equation \eqref{def of V-pi}. By the uniqueness of \eqref{eq:pe-V} in $C_b(\mathbb R^d)
\cap
C_{\mathrm{loc}}^2(\mathbb R^d)$, this completes the proof of
\[
\E_{a\sim\pi(\cdot\mid s)}[Q^\pi(s,a)] = V^\pi(s),
\]
and the claim that the continuous-time Bellman equation admits a unique solution for which $\mathbb{E}_{a\sim\pi}[Q^\pi(s,a)] \in
C_b(\mathbb R^d)
\cap
C_{\mathrm{loc}}^2(\mathbb R^d).$

\item Denote by $P_{\Delta t}^a$ the frozen-action transition operator,
\[
P_{\Delta t}^af(s)
:=
\mathbb E\left[
f(s_{\Delta t})
\mid
s_0=s,\ 
a_t=a\text{ for }0\leq t\leq\Delta t
\right].
\]
The discrete-time Bellman equation gives $Q_{\Delta t}^\pi(s,a)
=
\Delta t\,r(s,a)
+
\gamma P_{\Delta t}^aV_{\Delta t}^\pi(s).$
It follows that
\[
\begin{aligned}
\frac{Q_{\Delta t}^\pi(s,a)-V_{\Delta t}^\pi(s)}
{1-\gamma}
+V_{\Delta t}^\pi(s)
&=
\frac{\Delta t}{1-\gamma}
\left[
r(s,a)
+
\gamma
\frac{
P_{\Delta t}^aV_{\Delta t}^\pi(s)
-
V_{\Delta t}^\pi(s)}
{\Delta t}
\right].
\end{aligned}
\]
By the convergence of the discrete-time policy-evaluation equation and the short-time generator consistency of the frozen-action transition operator,
\[
V_{\Delta t}^\pi
\longrightarrow
V^\pi,
\qquad
\frac{
P_{\Delta t}^aV_{\Delta t}^\pi
-
V_{\Delta t}^\pi}
{\Delta t}
\longrightarrow
\mathcal L_{b,\Sigma}V^\pi
\]
as $\Delta t\to0$. These results follow from the standard semigroup expansion and Dynkin's formula; see \citep{jia2023q,tallec2019making,mou2024bellman}. Moreover, $\frac{\Delta t}{1-\gamma}
\longrightarrow
\frac{1}{\beta}$. Therefore,
\[
\begin{aligned}
\lim_{\Delta t\to0}
\left[
\frac{
Q_{\Delta t}^\pi(s,a)
-
V_{\Delta t}^\pi(s)}
{1-\gamma}
+
V_{\Delta t}^\pi(s)
\right]
&=
\frac{1}{\beta}
\left[
r(s,a)+(\mathcal L_{b,\Sigma} V^\pi)(s)
\right].
\end{aligned}
\]
Because we know from part 1, that $V^\pi = \mathbb{E}_a[Q^\pi(s,a)]$, so one  completes the proof of part 2. 

\item 
Due to the definition of $\pi'(s) = \argmax Q^\pi(s,a)$, using part~1 gives
\[
Q^\pi(s,\pi'(s))
\geq
\E_a[Q^\pi(s,a)]
=
V^\pi(s).
\]
As $Q^\pi(s,a)$ is the solution to \eqref{def of Q-pi}, by setting $a = \pi'$, one has
\[
r^{\pi'}(s)+(\mathcal L^{\pi'}V^\pi)(s)
-\beta V^\pi(s)
\geq0.
\]
Applying the discounted Dynkin formula to $V^\pi$ along the diffusion
controlled by $\pi'$ and then letting the terminal time tend to infinity
yields
\[
V^{\pi'}(s)-V^\pi(s)
=
\mathbb E_s^{\pi'}
\left[
\int_0^\infty e^{-\beta t}
\left(
r^{\pi'}+\mathcal L^{\pi'}V^\pi-\beta V^\pi
\right)(s_t)\,dt
\right]
\geq0.
\]
The discounted Dynkin formula and this standard policy-improvement
argument can be found in \citep{FS06,YZ99,P09}. Hence
$V^{\pi'}(s)\geq V^\pi(s)$ for every $s$.
\end{enumerate}
\end{proof}

\section{Proof of Theorem \ref{thm:q-star-property}}\label{sec:proof of thm-q-star-property}
\begin{proof}
\begin{enumerate}
    \item 
Let $V^*$ be the unique solution of the HJB equation
\eqref{eq:hjb} in
$C_b(\mathbb R^d)\cap C_{\mathrm{loc}}^2(\mathbb R^d)$, and define
\begin{equation}\label{temp-def}
\hat{Q}^*(s,a)
:=
\frac{1}{\beta}
\left[
r(s,a)+\mathcal L_{b, \Sigma}V^*(s)
\right].
\end{equation}
By the HJB equation \eqref{eq:hjb}, one has
\begin{equation}\label{eq:temp-vstar}
    \max_{a\in\mathcal A}\hat{Q}^*(s,a)
=
\frac{1}{\beta}
\max_{a\in\mathcal A}
\left\{
r(s,a)+\mathcal L_{b, \Sigma}V^*(s)
\right\}=
V^*(s).
\end{equation}
First, replacing $V^*(s)$ in the above equation by $\max_{a\in\mathcal A}\hat{Q}^*(s,a)$, one finds that $\max_{a\in\mathcal A}\hat{Q}^*(s,a)$ solves the HJB equation \eqref{eq:hjb}.  By the uniqueness of the HJB equation in $C_b(\mathbb R^d)\cap C_{\mathrm{loc}}^2(\mathbb R^d)$, one has $\max_{a\in\mathcal A}\hat{Q}^*(s,a) = V^*(s)$

Second, replacing $V^*(s)$ by $\max_{a\in\mathcal A}\hat{Q}^*(s,a)$ in \eqref{temp-def} implies that $\hat{Q}^*(s,a)$ satisfies the continuous-time optimal Bellman equation \eqref{pde of Qstar}. Therefore, $\hat{Q}^*(s,a) = Q^*(s,a)$, which completes the proof for part 1 and the statement that the PDE \eqref{pde of Qstar} admits a unique solution $Q^*$ among functions $Q$ for which $\max_{a\in\mathcal A}Q^*(s,a) \in
C_b(\mathbb R^d)
\cap
C_{\mathrm{loc}}^2(\mathbb R^d).$

\item
As proved in part 1, 
$\max_aQ^*(s,a)=V^*(s)$, and by the definition of $\pi^*(s) = \argmax_a Q^*(s,a)$, one has
\[
Q^*(s,\pi^*(s))=V^*(s).
\]
Using the definition of $Q^*$, this identity is equivalent to
\[
\beta V^*(s)
=
r(s,\pi^*(s))
+
\mathcal L_{b^{\pi^*}, \Sigma^{\pi^*}}V^*(s).
\]
Thus, $V^*$ solves the policy-evaluation equation associated with
$\pi^*$. Since $\pi^*$ is admissible, this equation admits $V^{\pi^*}$
as its unique bounded classical solution. Therefore,
\[
V^{\pi^*}=V^*.
\]
In addition, the
verification theorem \citep{FS06,YZ99,P09} gives
\[
V^\pi(s)\leq V^*(s)
\qquad
\text{for every }\pi\in\Pi_{\mathrm{ad}},
\]
and hence
\[
V^{\pi^*}(s)
=
V^*(s)
=
\max_{\pi\in\Pi_{\mathrm{ad}}}V^\pi(s).
\]

Finally, by Theorem~\ref{thm:q-pi-property},
\[
Q^{\pi^*}(s,a)
=
\frac{1}{\beta}
\left[
r(s,a)+\mathcal L_{b, \Sigma}V^{\pi^*}(s)
\right].
\]
Since $V^{\pi^*}=V^*$ and the conclusion in part 1, it follows that
\[
Q^{\pi^*}(s,a)
=
\frac{1}{\beta}
\left[
r(s,a)+\mathcal L_{b, \Sigma}V^*(s)
\right]
=
\hat{Q}^*(s,a) = Q^*(s,a).
\]
\end{enumerate}
\end{proof}

\section{Additional Figures}\label{appendix:figures}

\begin{figure*}[t]
    \centering
    \includegraphics[width=\textwidth]{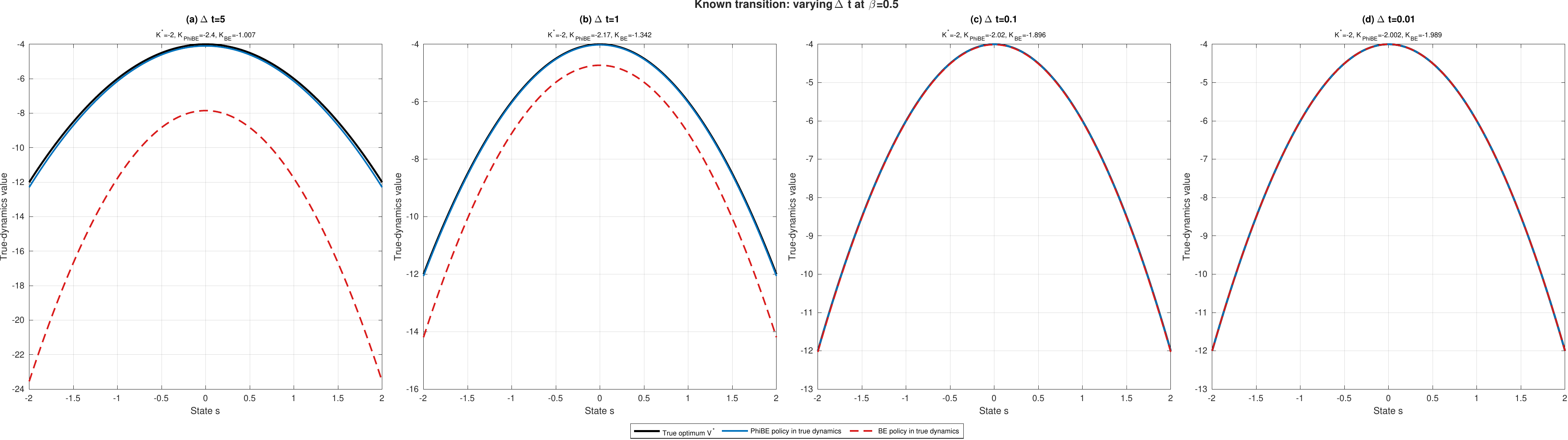}
    \caption{
    True-dynamics value-function comparison as the observation interval
    varies, with $\beta=0.5$. The four panels correspond to
    $\dt=5$, $1$, $0.1$, and $0.01$. Each panel compares the value
    generated by the true continuous-time optimal policy with the values
    generated by the optimal policies of the PhiBE and BE surrogate
    models, all evaluated in the original continuous-time dynamics.
    }
    \label{fig:known-transition-dt-value-appendix}
\end{figure*}

\begin{figure*}[t]
    \centering
    \includegraphics[width=\textwidth]{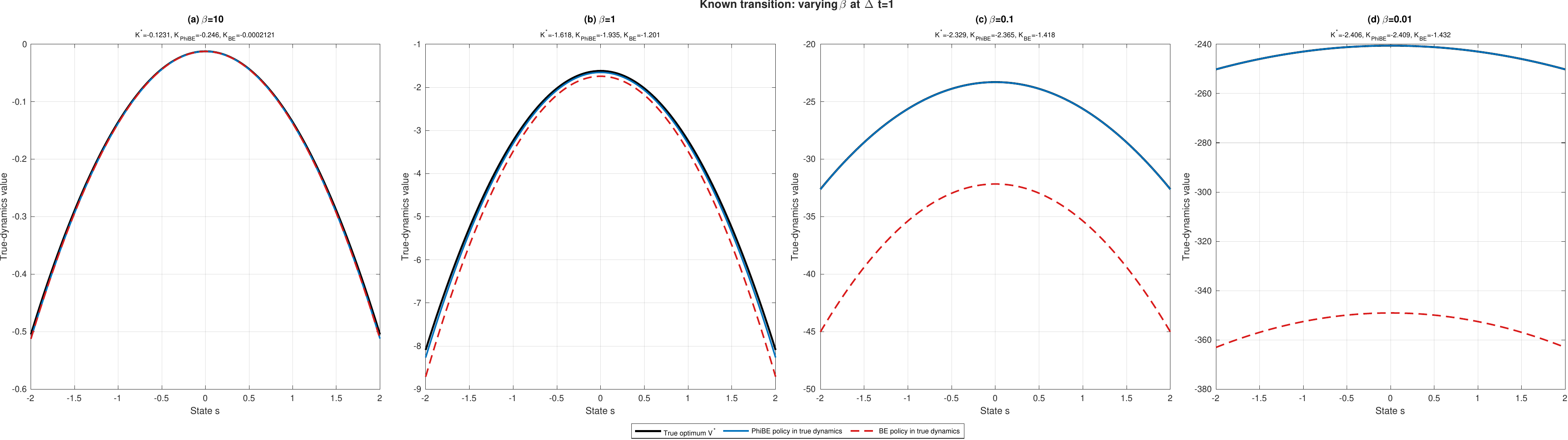}
    \caption{
    True-dynamics value-function comparison as the discount rate varies,
    with $\dt=1$. The four panels correspond to
    $\beta=10$, $1$, $0.1$, and $0.01$. Each panel compares the value
    generated by the true continuous-time optimal policy with the values
    generated by the optimal policies of the PhiBE and BE surrogate
    models, all evaluated in the original continuous-time dynamics.
    }
    \label{fig:known-transition-beta-value-appendix}
\end{figure*}

\bibliography{PDL.bib}

\end{document}